\documentclass[11pt,a4paper]{amsart}
\usepackage{amssymb}
\usepackage{latexsym}
\usepackage{exscale}
\usepackage{amsfonts}
\usepackage{graphicx}
\usepackage{mathrsfs}
\usepackage{amsmath,amscd,amsthm}
\usepackage{bbm,pifont}
\usepackage{enumerate}
\usepackage{color}
\usepackage{amssymb}
\usepackage{latexsym}
\usepackage{exscale}

\allowdisplaybreaks

\numberwithin{equation}{section}
\newtheorem{theorem}{Theorem}[section]
\newtheorem{lemma}[theorem]{Lemma}
\newtheorem{corollary}[theorem]{Corollary}

\newtheorem{definition}[theorem]{Definition}

\newtheorem{remark}[theorem]{Remark}

\newcommand{\supp}{\operatorname{supp}}

\def\supp{\mathop{\rm supp}}

\def\X{{\mathrm{X}}}

\begin{document}
\allowdisplaybreaks

\title[Compactness of Riesz transform commutator on stratified Lie groups]
{Compactness of Riesz transform commutator on stratified Lie groups}
\author{Peng Chen, Xuan Thinh Duong, Ji Li and Qingyan Wu}

\address{Peng Chen, Department of Mathematics, Sun Yat-Sen University, Guangzhou, 510275, China}
\email{chenpeng3@mail.sysu.edu.cn}

\address{Xuan Thinh Duong, Department of Mathematics, Macquarie University, NSW, 2109, Australia}
\email{xuan.duong@mq.edu.au}

\address{Ji Li, Department of Mathematics, Macquarie University, NSW, 2109, Australia}
\email{ji.li@mq.edu.au}

\address{Qingyan Wu, Department of Mathematics\\
         Linyi University\\
         Shandong, 276005, China
         }
\email{wuqingyan@lyu.edu.cn}


\begin{abstract}
Let $\mathcal G$ be a stratified  Lie group and $\{\X_j\}_{1 \leq j \leq n}$ a basis for the left-invariant vector fields of degree one on $\mathcal G$.
Let $\Delta = \sum_{j = 1}^n \X_j^2 $ be the sub-Laplacian
on $\mathcal G$. The $j^{\mathrm{th}}$  Riesz transform on $\mathcal G$  is defined by $R_j:= \X_j (-\Delta)^{-\frac{1}{2}}$,
 $1 \leq j \leq n$. 
In this paper, we provide a concrete construction of the ``twisted truncated sector'' which is related to the pointwise lower bound of the kernel of $R_j$ on $\mathcal G$. Then 
 we obtain the characterisation of compactness of the commutators of  $R_j$ with respect to VMO, the space of functions with vanishing mean oscillation on $\mathcal G$.
\end{abstract}

\maketitle


\section{Introduction and statement of main results}
\setcounter{equation}{0}


A central topic of modern harmonic analysis is to study singular integral operators and their applications in characterizing function spaces.
In \cite{Cal}, Calder\'on introduced the commutator of a singular integral operator $T$ with a symbol $b$ as
$$[b,T](f):=bT(f)-T(bf).$$
When $T$ is the Riesz transform $R_j={\partial\over\partial x_j}\Delta^{-{1\over2}}$ on the Euclidean space $\mathbb R^n$, {Coifman, Rochberg and Weiss  \cite{CRW}  showed that the commutator $[b,R_j]$ is bounded on $L^p(\mathbb R^n)$ with $1<p<\infty$ if and only if $b\in {\rm BMO}(\mathbb R^n)$, which is the space of functions with bounded mean oscillation.  
}
Uchiyama \cite{U1} then showed that $[b,R_j]$ is compact on $L^p(\mathbb R^n)$ with $1<p<\infty$ if and only if $b\in {\rm VMO}(\mathbb R^n)$, the space of functions with vanishing mean oscillation on $\mathbb R^n$.
\color{black}
 Later on and recently, there has been an intensive study of the compactness of commutators of singular integrals in many different settings, such as the Riesz transform associated with Bessel operator on the positive real line, the Cauchy's integrals on the real line, the Calder\'on--Zygmund operator associated with homogeneous kernels ${\Omega(x)\over |x|^n}$ on $\mathbb R^n$, and the multilinear Riesz transforms,  see for example \cite{CCHTW, DLMWY, GWY, KL1,KL2, LNWW,TYY} and related references therein.

Beyond these operators in the Euclidean setting, it is natural to ask whether this characterisation of compactness of commutators also holds for Riesz transforms associated with the sub-Laplacian on Heisenberg groups $\mathbb H^n$, which is the boundary of the Siegel upper half space in $\mathbb C^n$.  Recall that $\mathbb H^n$ and the Siegel upper half space are holomorphically equivalent to the unit sphere and unit ball in $\mathbb C^n$, and hence the role of Riesz transform associated with the sub-Laplacian on  $\mathbb H^n$ is similar to the role of Hilbert transform on the real line.

We note that to study the boundedness and compactness of Riesz transform commutator, one only needs the upper bound of the Riesz transform kernel and its derivative (see \cite{FS},\cite{Sa}). However, a full characterisation of the Riesz transform commutator would also require the kernel lower bound. 
Recently, in \cite{DLLW,DLLWW} the authors studied the pointwise lower bound of the kernel of Riesz transform
associated with the sub-Laplacian on stratified Lie groups, and then they established the characterisation of commutator of  Riesz transforms with respect to the BMO space (see Theorem 1.2 in \cite{DLLW} and Theorems 1.2--1.5 in \cite{DLLWW}), which extends the classical result of Coifman--Rochberg--Weiss \cite{CRW} to the setting of stratified Lie groups. One of the important examples of stratified Lie groups is the Heisenberg group $\mathbb H^n$.   And the pointwise kernel lower bound obtained in \cite{DLLWW} is as follows. 

\smallskip
\noindent {\bf Theorem A\ }(\cite{DLLWW}, Theorem 1.1){\bf.} {\it
Suppose that $\mathcal G$ is a stratified Lie group with homogeneous dimension $Q$ and that $j\in\{1,2,\ldots,n\}$.  
There exist a large positive constant $r_o$ and a positive constant $C$  such that for every 
$g\in \mathcal G$ there exists a ``twisted truncated sector'' $G\subset \mathcal G$ such that $\inf\limits_{g'\in G}\rho(g,g')=r_o$ and that  for every $g_1\in B_\rho(g,1)$ and $g_2\in G$, we have
\begin{align*}
 |K_j(g_1,g_2)|\geq C  \rho(g_1,g_2)^{-Q}, \quad |K_j(g_2,g_1)|\geq C  \rho(g_1,g_2)^{-Q},
\end{align*}
and all $K_j(g_1,g_2)$ as well as all $K_j(g_2,g_1)$ have the same sign.

Moreover, this ``twisted truncated sector''  $G$ is regular, in the sense that  $|G|=\infty$ and that for any $R>2r_o$,
\begin{align}\label{regular0}
 |B_\rho(g,R) \cap G|\approx R^Q,  
\end{align}
where the implicit constants are independent of $g$ and $R$. 
}

 Here and in what follows, $\rho$ is the homogeneous norm on $\mathcal G$, and for $g\in\mathcal G$, $r>0$, $B_\rho(g,r)$ is the ball defined via $\rho$, and $K_j(g_1,g_2)$ is the kernel of the $j$th Riesz transform $R_j$. For the details of the notation, we refer to Section 2 below.

The aim of this paper is to establish the characterisation of the compactness of the commutator of Riesz transform associated with sub-Laplacian on stratified Lie groups via 
{the VMO space, which is defined as the closure of the $C_0^\infty$ functions (functions with arbitrary order of derivatives and with compact support) under the norm of the BMO space.} 
For the precise definition of $C_0^\infty$ functions, the properties of Riesz transforms and the BMO, Hardy spaces on stratified Lie groups, we refer to Folland--Stein \cite{FS}, see also Saloff-Coste \cite{Sa}. See Theorem \ref{thm-cmo} for the equivalent characterisation of VMO as the space of functions with vanishing mean oscillation.

However, 
to establish the characterisation of compactness of Riesz commutators, 
 we point out that the condition \eqref{regular0} for the ``twisted truncated sector'' $G\subset \mathcal G$ related to the pointwise lower bound of the kernel of $R_j$ is not enough, since we need to know more about the behaviour of this twisted truncated sector $G$ in each annuli that intersects with $G$.

Thus, the main results of this paper are twofold. First, we give a particular construction of the ``twisted truncated sector'' $G\subset \mathcal G$ related to the pointwise lower bound of the kernel of $R_j$, which is regular in each annuli that intersects with $G$, while the previous version in Theorem A only states that the ``twisted truncated sector'' is regular in each large ball that intersects with $G$.
Second, by using this kernel lower bound and the more explicit information on $G$, we establish the characterisation of the compactness of Riesz commutators via functions in VMO space on $\mathcal G$, where a characterisation of the VMO space is needed.

To be more precise, we have the following results.

 \begin{theorem}\label{thm0}
 Suppose that $\mathcal G$ is a stratified  Lie group with homogeneous dimension $Q$ and that $j\in\{1,2,\ldots,n\}$.  
There exist a large positive constant $r_o$ and a positive constant $C$ such that for every 
$g\in \mathcal G$ there exists a ``twisted truncated sector'' $G_g\subset \mathcal G$ satisfying that  $\inf\limits_{g'\in G_g}\rho(g,g')=r_o$ and that  for every $g_1\in B_\rho(g,1)$ and $g_2\in G_g$, we have
$$ |K_j(g_1,g_2)|\geq C \rho(g_1,g_2)^{-Q}, \quad |K_j(g_2,g_1)|\geq C \rho(g_1,g_2)^{-Q},$$
and all $K_j(g_1,g_2)$ as well as all $K_j(g_2,g_1)$ have the same sign.

Moreover, this ``twisted truncated sector''  $G_g$ is regular, in the sense that  $|G_g|=\infty$ and that for any $R_2>R_1>2r_o$,
\begin{align}\label{regular}
 \left|\left(B_\rho(g,R_2)\setminus B_\rho(g,R_1)\right) \cap G_g\right|\approx  \left|B_\rho(g,R_2)\setminus B_\rho(g,R_1)\right|,  
 \end{align}
where the implicit constants are independent of $g$ and $R_1, R_2$. 

\end{theorem}

Here we point out that the set $G_g$ that we constructed in Theorem \ref{thm0} above 
is a connected open set spreading out to infinity, which plays the role of 
the ``truncated sector centred at a fixed point'' in the Euclidean setting. 
The shape of $G_g$ here may not be the same as the usual sector since the norm $\rho$ 
on $\mathcal G$ is different from the standard Euclidean metric. However, such a kind of twisted sector always exists.

Second, based on the property of the Riesz transform kernel, we establish the following commutator theorem on stratified Lie group via providing the characterisation of the VMO space, following the approach of Uchiyama \cite{U1}. In what follows we use $A_p(\mathcal G)$ to denote the Muckenhoupt type weighted class on stratified Lie groups, whose precise definition will be given in Section 2.
\begin{theorem}\label{thm1}
Let $1<p<\infty$, $w\in A_p(\mathcal G)$, $b\in L^1_{loc}(\mathcal G)$. Then $b\in {\rm VMO}(\mathcal G)$ if and only if for some $\ell\in\{1,\cdots, n\}$, Riesz transform commutator $[b, \mathcal R_{\ell}]$ is compact on $L_w^{p}(\mathcal G)$.
\end{theorem}

This paper is organised as follows. In Section 2 we recall necessary preliminaries on stratified nilpotent Lie groups $\mathcal G$. In Section 3 we provide a particular construction of the twisted truncated sector and then obtain the pointwise lower bound of the Riesz transform kernels, and then prove Theorem \ref{thm0}. In Section 4, by using the kernel lower bound that
we established, we prove Theorem \ref{thm1}, the characterisation of compactness of the Riesz commutator. 
In the end, in the appendix, we provide the characterisation of the VMO space following the approach of Uchiyama \cite{U1}.

\vspace{0.2cm} 

{

{\bf Notation:} Throughout this paper, $\mathbb{N}$ will denote the set of all nonnegative integers. For a real number $a$, $[a]$ means the largest integer no greater than $a$.
In what follows, $C$ will denote positive constant which is independent of the main parameters, but it may vary from line to line. By $f\lesssim g$,
we shall mean $f\le C g$ for some positive constant 
 $C$. If $f\lesssim g$ and $g\lesssim f$, we then write $f \approx g$.
}



\section{Preliminaries on stratified Lie groups $\mathcal G$}
\setcounter{equation}{0}

Recall that a connected, simply connected nilpotent Lie group $\mathcal{G}$ is said to be stratified if its left-invariant Lie algebra $\mathfrak{g}$ (assumed real and of finite dimension) admits a direct sum decomposition
\begin{align*}
\mathfrak{g} = \bigoplus_{i = 1}^s V_i , \quad [V_1, V_i] = V_{i + 1}, \  \mbox{ for $i \leq s - 1$ and $[V_1, V_s]=0$.}
\end{align*}
$s$ is called the step of the group $\mathcal G$.

For $i=1,\cdots,s$, let $n_i=\dim V_i$ and $m_i= n_1+\cdots+n_i$, $m_0=0$ and $m_s=N$.We choose once and for all a basis $\{X_1,\cdots, X_N\}$ for $\mathfrak{g}$ adappted to the stratification, that is, such that
$\{X_{m_{j-1}+1},\cdots, X_{m_j}\}$ is a basis of $V_j$ for each $j=1,\cdots,s$. 
One identifies $\mathfrak{g}$ and $\mathcal{G}$ via the exponential map
\begin{align*}
\exp: \mathfrak{g} \longrightarrow \mathcal{G},
\end{align*}
which is a diffeomorphism, {i.e., any $g\in\mathcal{G}$ can be written in a unique way as $g=\exp(x_1X_1+\cdots+x_NX_N)$. Using these exponential coordinates, we identify $g$ with the $N$-tuple $(x_1,\cdots,x_N)\in\mathbb{R}^N$ and identify $\mathcal G$ with $(\mathbb{R}^N, \circ)$, where the group operation $\circ$ is determined by the Campbell-Hausdorff formula (c.f. \cite[Section 2.2.2]{BLU})}.

We fix once and for all a (bi-invariant) Haar measure $dx$ on $\mathcal G$ (which is just the lift of Lebesgue measure on $\frak g$ via $\exp$).

There is a natural family of dilations on $\frak g$ defined for $r > 0$ as follows:
\begin{align*}
\delta_r \bigg( \sum_{i = 1}^s v_i \bigg) = \sum_{i = 1}^s r^i v_i, \quad \mbox{with $v_i \in V_i$}.
\end{align*}
This allows the definition of dilation on $\mathcal{G}$, which we still denote by $\delta_r$ (see Section \ref{S3}).

Denote by $n=n_1$, for the basis $\{\X_1, \cdots, \X_n\}$ of $V_1$, we consider the sub-Laplacian $\Delta = \sum_{j=1 }^n \X_j^2 $. Observe that $\X_j$ ($1 \leq j \leq n$) is homogeneous of degree $1$ with respect to the dilations, and $\Delta$ of degree $2$ in the sense that :
\begin{align*}
&\X_j \left( f \circ \delta_r \right) = r \, \left( \X_j f \right) \circ \delta_r, \qquad  1 \leq j \leq  n, \  r > 0, \ f \in C^1, \\
&\delta_{\frac{1}{r}} \circ \Delta \circ \delta_r = r^2 \, \Delta, \quad \forall r > 0.
\end{align*}

Let $Q$ denote the homogeneous dimension of $\mathcal{G}$, namely,
\begin{align}\label{homo dimension}
Q= \sum_{i=1}^s i\, {\rm dim} V_i.
\end{align}
And let $p_h$ ($h > 0$) be the heat kernel (that is, the integral kernel of $e^{h \Delta}$)
on $\mathcal G$. For convenience, we set $p_h(g) = p_h(g, 0)$ (that is, in this note, for a convolution operator, we will identify the integral kernel with the convolution kernel) and $p(g) = p_1(g)$.

Recall that (c.f. for example \cite{FS})
\begin{align} \label{hkp1}
p_h(g) = h^{-\frac{Q}{2}} p(\delta_{\frac{1}{\sqrt{h}}}(g)), \qquad  \forall h > 0, \  g \in \mathcal G.
\end{align}

The kernel of the $j^{\mathrm{th}}$  Riesz transform $\X_j (-\Delta)^{-\frac{1}{2}}$ ($1 \leq j \leq  n$) is written simply as $K_j(g, g') = K_j(g'^{-1} \circ g)$. It is well-known that
\begin{align} \label{kjs}
K_j \in C^{\infty}(\mathcal G \setminus \{0\}), \ K_j(\delta_r(g)) = r^{-Q} K_j(g), \quad \forall g \neq 0, \ r > 0, \ 1 \leq j \leq  n,
\end{align}
which also can be explained by \eqref{hkp1} and the fact that
\begin{align*}
K_j(g) = \frac{1}{\sqrt{\pi}} \int_0^{+\infty} h^{-\frac{1}{2}} \X_j p_h(g) \, dh  = \frac{1}{\sqrt{\pi}} \int_0^{+\infty} h^{- \frac{Q}{2} - 1} \left( \X_j p \right)(\delta_{\frac{1}{\sqrt{h}}}(g)) \, dh.
\end{align*}

Next we recall the homogeneous norm $\rho$ (see for example \cite{FS}) on $\mathcal G$
which is defined to be a continuous function
$g\to \rho(g)$ from $\mathcal G$ to $[0,\infty)$, which is $C^\infty$ on $\mathcal G\backslash \{0\}$
%
and satisfies
\begin{enumerate}
\item[(a)] $\rho(g^{-1}) =\rho(g)$;
\item[(b)] $\rho({ \delta_r(g)}) =r\rho(g)$ for all $g\in \mathcal G$ and $r>0$;
\item[(c)] $\rho(g) =0$ if and only if $g=0$.
\end{enumerate}
For the existence (also the construction) of the homogeneous norm $\rho$ on  $\mathcal G$,
we refer to  \cite[Chapter 1, Section A]{FS}. { For convenience,
we set
\begin{align*}
\rho(g, g') = \rho(g'^{-1} \circ g) = \rho(g^{-1} \circ g'), \quad \forall g, g' \in \mathcal{G}.
\end{align*}
Recall that (see \cite{FS}) this defines a quasi-distance in sense of  Coifman-Weiss, namely, there exists a constant $C_\rho > 0$ such that
 \begin{align} \label{qdr}
 \rho(g_1, g_2) \leq C_{\rho} \, \left( \rho(g_1, g') + \rho(g', g_2)   \right), \qquad \forall g_1, g_2, g' \in  \mathcal{G}.
\end{align}
In the sequel, we fix a homogeneous norm $\rho$ on  $\mathcal G$ (see Section \ref{S3}).

We now denote by $d$ the Carnot--Carath\'eodory metric associated to $\{\X_j\}_{1\leq j\leq n}$, which is equivalent to $\rho$ in the sense that: there exist $C_{d_1},C_{d_2}>0$ such that
for every $g_1,g_2\in\mathcal G$ (see \cite{BLU}),
\begin{align} \label{equi metric d}
 C_{d_1}\rho(g_1,g_2)\leq d(g_1,g_2)\leq C_{d_2} \rho(g_1,g_2).
\end{align}
We point out that {the Carnot--Carath\'eodory metric $d$ even on the most special stratified Lie group, the Heisenberg group, is not smooth on $\mathcal G \setminus \{0\}$.}

In the sequel, to avoid confusion we will use $B_\rho(g,r)$ and $S_\rho(g,r)$  to denote the open ball and the sphere with center $g$ and radius $r$ defined by $\rho$, respectively.  And we will use $B(g,r)$ and $S(g,r)$ to denote the open ball and the sphere defined by $d$, respectively.  In the following, $B$ is always a ball defined by $d$ and $r_B$ is its radius.  For any $\alpha>0$, denote by $\alpha B(g,r)=B(g,\alpha r)$.

\begin{definition}
The bounded mean oscillation space $\operatorname{BMO}(\mathcal G)$ is defined to be the space of all locally integrable functions $f$ on $\mathcal G$ such that
$$\|f\|_{{\rm BMO}(\mathcal G)}:=\sup_{B\subset \mathcal G}M(f,B):=\sup_{B\subset \mathcal G}{1\over |B|}\int_{B}\left|f(g)-f_B\right|dg<\infty,$$
where 
\begin{align}\label{fave}
f_B={1\over |B|}\int_{B}f(g)dg.
\end{align}
\end{definition}
\begin{definition}
We define ${\rm VMO}(\mathcal G)$ as the closure of the $C_0^\infty$ functions on $\mathcal G$
 under the norm of the BMO space.
\end{definition}


\color{black}
For the Folland--Stein BMO space  ${\rm BMO}(\mathcal G)$,  note that we have an {equivalent norm}, 
which is 
defined by
 $$\|b\|'_{\operatorname{BMO}(\mathcal G)}=\sup_{B\subset\mathcal G}\inf_{c}\frac{1}{|B}|\int_{B}|b(g)-c|dg.$$
 {For a ball B}, the infimum above is attained and the constants where this happens can be found among the median values.
 
 \begin{definition}\label{medianv}
 A {median value} of $b$ over a ball $B$ will be any real number $m_b(B)$ that satisfies simultaneously
 $$\left|\{x\in B: b(g)>m_b(B)\}\right|\leq \frac{1}{2}|B|$$
 and
 $$\left|\{x\in B: b(g)<m_b(B)\}\right|\leq \frac{1}{2}|B|.$$
 \end{definition}
Following the standard proof in \cite[p.199]{To}, we can see that the constant $c$ in the definition of $\|b\|'_{\operatorname{BMO}(\mathcal G)}$ can be chosen to be a median value of $b$. And it is easy to see that for any ball $B\subset \mathcal G$, 
\begin{align}\label{mmedian}
M(b,B)\approx {1\over |B|}\int_{B}\left|f(g)-m_b(B)\right|dg,
\end{align}
where the implicit constants are independent of the function $b$ and the ball $B$. 
 
The theory of $A_{p}$ weight was first introduced by Muckenhoupt in the study of
weighted $L^{p}$ boundedness of Hardy-Littlewood maximal functions in \cite%
{Mu}. For $A_p$ weights on the stratified Lie group (which is an example of spaces of homogeneous type in the sense of Coifman and Weiss \cite{CW}) one can refer to \cite{HPR}. 
By a weight, we mean a non-negative locally integrable function on $\mathcal G$.

\begin{definition}
\ Let $1<p<\infty $, a  weight $w$ is said to be of class  $A_{p}(\mathcal G)$ if
\begin{equation*}
[w]_{A_p}:=\sup_{B\subset\mathcal G}\left( \frac{1}{|B|}\int_{B}w(g)dg\right) \left( \frac{1}{|B|}%
\int_{B}w(g) ^{-1/(p-1)}dg\right) ^{p-1}<\infty.
\end{equation*}%
A weight $w$ is said to be of class  $A_{1}(\mathcal G)$ if there exists a constant $C$
such that for all balls $B\subset\mathcal G$,
\begin{equation*}
\frac{1}{|B|}\int_{B}w( g) dg\leq C\mathop{\rm essinf}%
\limits_{x\in B}w(g) .
\end{equation*}%
For $p=\infty$, we define
\begin{equation*}
A_{\infty }(\mathcal G)= \bigcup_{1\leq p<\infty }A_{p}(\mathcal G).
\end{equation*}
\end{definition}
Note that $w$ is doubling when it is in $A_p$, i.e. there exists a positive constant $C$ such that $w(2B)\leq C w(B)$ for every ball $B$.

Recall that the Muckenhoupt weights have some fundamental properties. 
A close relation to $A_{\infty }(\mathcal G)$ is the reverse H\"{o}lder
condition. If there exist $r>1$ and a fixed constant $C$ such that
\begin{equation*}
\left( \frac{1}{|B|}\int_{B}w(g)^{r}dg\right) ^{1/r}\leq \frac{C}{|B|}%
\int_{B}w(g)dg
\end{equation*}%
for all balls $B\subset \mathcal G$, we then say that $w$ satisfies the
\textit{reverse H\"{o}lder condition of order $r$} and write $w\in RH_{r}(%
\mathcal G)$. According to \cite[Theorem 19 and Corollary 21]{IMS}, $ w\in A_{\infty }(\mathcal G)$ if and only if there exists some $r>1$ such
that $w\in RH_{r}(\mathcal G)$. 

For any $w\in A_{\infty }(\mathcal G)$ and any Lebesgue
measurable set $E$, denote by $w(E):=\int_{E}w(g)dg$. 
By the definition of $A_p$ weight and H\"older's inequality, we can easily obtain the  following
standard properties.
\begin{lemma}\label{lemlw}
 Let $w\in A_{p}(\mathcal G)\cap RH_{r}(\mathcal G), p\geq1$ and $r>1$. Then there
exist constants $C_{1}, C_{2}>0$ such that
\begin{equation*}
C_{1}\left(\frac{|E|}{|B|}\right)^{p}\leq\frac{w(E)}{w\left(B\right)}\leq
C_{2}\left(\frac{|E|}{|B|}\right)^{(r-1)/r}
\end{equation*}
for any measurable subset $E$ of a ball $B$. Especially, for any $\lambda >1$,
\begin{equation*}
w\left( B\left( g_{0},\lambda R\right) \right) \leq C\lambda ^{Qp}w\left(
B\left( g_{0},R\right) \right),
\end{equation*}
where $Q$ is the homogeneous dimension of $\mathcal G$.
\end{lemma}

\section{Lower bound for kernel of Riesz transform $\mathcal R_j:= \X_j (-\Delta)^{-\frac{1}{2}}$ and\\ proofs of Theorems \ref{thm0}}\label{S3}
\setcounter{equation}{0}

In this section, we study a suitable version of the lower bound for kernel of Riesz transform $\mathcal R_j:= \X_j (-\Delta)^{-\frac{1}{2}}$, $j=1,\ldots,n$, on stratified Lie group $\mathcal G$. Here we will use the Carnot--{Carath\'eodory} metric $d$ associated to $\{\X_j\}_{1\leq j\leq n}$ to study the lower bound, and we also  make good use of the dilation structure on $\mathcal G$. It is not clear whether one can obtain similar lower bounds for the Riesz kernel on general nilpotent Lie groups which is not stratified.

{To begin with, we first recall that
by the classical estimates for heat kernel and its derivations on stratified Lie groups {(see for example \cite{Sa,VSC92})}, it is well-known that for any $1 \leq j \leq n$ and $g\neq g'$
\begin{align} \label{kj}
|K_j(g, g')| + d(g, g') \sum_{i = 1}^n \left( |\X_{i, g} K_j(g, g')| + |\X_{i, g'} K_j(g, g')| \right)
\lesssim  d(g, g')^{-Q}, 
\end{align}
where $\X_{i, g}$ denotes the derivation with respect to $g$.

Two important families of diffeomorphisms of $\mathcal G$ are the translations and dilations of $\mathcal G$. For any $g\in \mathcal G$, the (left) translation $\tau_g : \mathcal{G}\rightarrow \mathcal{G}$ is defined as
$$\tau_g (g')=g\circ g'.$$
 For any $\lambda>0$, the dilation $\delta_{\lambda}:\mathcal{G}\rightarrow \mathcal{G}$, is defined as
\begin{equation}\label{delta}
 \delta_{\lambda}(g)=\delta_{\lambda}\left(x_1,x_2,\cdots,x_N\right)=\left(\lambda^{\alpha_1} x_1,\lambda^{\alpha_2} x_2,\cdots,\lambda^{\alpha_N}x_N\right),
\end{equation}
where $\alpha_j=i$ whenever $m_{i-1}<j\leq m_i$, $i=1,\cdots,s$. Therefore, $1=\alpha_1=\cdots=\alpha_{n_1}<\alpha_{n_1+1}=2\leq\cdots\leq \alpha_n=s.$

 Before proving Theorem \ref{thm0}, we need the following elementary
properties of the group operation (see for example \cite{Mo}, \cite[Proposition 2.2.22]{BLU}).

\begin{lemma}\label{gp law}
The group law of $\mathcal G$ has the form
\begin{equation*}
g\circ g'=g+g'+P(g,g'),\quad \forall\  g, g'\in \mathbb{R}^N,
\end{equation*}
where $P=(P_1,P_2,\cdots,P_N):\mathbb{R}^N \times \mathbb{R}^N\rightarrow \mathbb{R}^N$ and each $P_j$ is a homogeneous polynomial of degree $\alpha_j$ with respect to the intrinsic dilations of $\mathcal G$ defined in \eqref{delta}, i.e.,
$$P_j(\delta(g),\delta(g'))=\lambda^{\alpha_j}P_j(g,g'),\quad \forall\  g,g'\in\mathcal G.$$
Moreover,
\begin{itemize}
\item[(i)] $P$ is anti-symmetric, i.e.,   for any $g,g'\in\mathcal G$, 
$P_j(g,g')=-P_j(-g',g).$
\item[(ii)] For any $g,g'\in\mathcal G$, 
$P_1(g,g')=\cdots=P_{n}(g,g')=0.$
\item[(iii)] For $n<j\leq N$, 
$P_j(g,0)=P_j(0,g'),\  P_j(g,g)=P_j(g,-g)=0.$
\item[(iv)] For any $g=(x_1,\cdots,x_N)$ and $g'=(y_1,\cdots,y_N)$, if $j\leq m_i$, $1<i\leq s$, 
$P_j(g,g')=P_j(x_1,\cdots,x_{m_{i-1}},y_1,\cdots,y_{m_{i-1}}).$
\item[(v)] 
$P_j(g,g')=\sum_{l,h}R^{i}_{l,h}(g,g')(x_l y_h-x_h y_l),$
where the functions $R^{i}_{l,h}$ are polynomials, homogeneous of degree $\alpha_i-\alpha_l-\alpha_h$ with respect to group dilations, and the sum is extended to all $l, h$ such that $\alpha_l+\alpha_h\leq \alpha_i$.
 \end{itemize}
\end{lemma}

\begin{remark}
It follows from Lemma \ref{gp law} that $\delta_{\lambda}:\mathcal{G}\rightarrow\mathcal{G}$ is an automorphism of the group, i.e.,
$$\delta_{\lambda}(g)\circ\delta_{\lambda}(g')=\delta_{\lambda}(g\circ g').$$
And the unit element of $\mathcal G$ is the origin $0=(0,\cdots,0)\in\mathbb{R}^N$. Consequently, the inverse $g^{-1}$ of an element $g=(x_1,\cdots, x_N)\in\mathcal{G}$ has the form 
$$g^{-1}=(-x_1,\cdots, -x_N).$$
\end{remark}

Define 
\begin{equation}\label{hom norm}
|g|_{\mathcal{G}}=\bigg(\sum_{j=1}^{s}\left|x^{(j)}\right|^{\frac{2s!}{j}}\bigg)^{1\over{2s!}},\quad g=(x^{(1)},\cdots, x^{(s)})\in\mathcal{G},
\end{equation}
where $|x^{(j)}|$ denotes the Euclidean norm on $\mathbb{R}^{n_j}$. Then $|\cdot|_{\mathcal{G}}$ is a homogeneous norm on $\mathcal G$ (see for example \cite[Section 5.1]{BLU}).
In what follows,  we will use $\rho(g)$ to denote  $|g|_{\mathcal{G}}$ for any $g\in\mathcal{G}$.

In \cite[Lemma 3.1]{DLLW},  the authors proved the following property for the Riesz kernel $K_j$. 
\begin{lemma}\label{lem non zero}
For all $1 \leq j \leq n$, we have $K_j \not\equiv 0$ in $\mathcal{G} \setminus \{ 0 \}$.
\end{lemma}

%

Now we prove Theorem \ref{thm0}.

\begin{proof}[Proof of Theorem \ref{thm0}]
For any fixed $ j\in\{1,\ldots,n\} $, by Lemma \ref{lem non zero} and the scaling property of $K_j$ (c.f. \eqref{kjs}), there exists a compact set $\Omega$ on the unit sphere $S_\rho(0,1)$ with $\sigma(\Omega)>0$, where $\sigma$ is the Radon measure on $S_\rho(0,1)$, satisfying 
$$ \rho(\tilde g)=1\quad{\rm and}\quad K_j(\tilde g)\not=0,\quad \forall\ \tilde{g}\in\Omega, $$
and all the values $K_j(\tilde g)$ on $\Omega$ have the same sign.

We claim that there exist $ 0<\varepsilon_o\ll1$ and $C(K_j)$ such that  for any $0 <\eta< \varepsilon_o$, any $\tilde g\in \Omega$ and for all  $g \in \mathcal G$ and $r > 0$, 
\begin{align}
 |K_j(g_1, g_2)| \geq C(K_j) r^{- Q}, \quad   |K_j(g_2, g_1)| \geq C(K_j) r^{- Q}.
\end{align}
for any $ g_1 \in B_\rho(g, \eta r),   g_2 \in B_\rho\left(g \circ \delta_r( \tilde g^{-1}), \eta r\right)$.
Moreover, all $K_j(g_1,g_2)$ and all $K_j(g_2,g_1)$ have the same sign.

In fact, for any fixed $\tilde {g}\in\Omega$, since $K_j$ is a $C^\infty$ function in $\mathcal G\backslash \{0\}$, there exists $ 0<\varepsilon_{\tilde{g}}\ll1$ such that
\begin{align}\label{non zero e1}
K_j(g')\not=0\quad {\rm and} \quad  |K_j(g')|>{1\over 2} |K_j(\tilde g)|
\end{align}
for all $g'\in B_\rho(\tilde g, 4C_\rho^2 \varepsilon_{\tilde{g}})$, where $C_\rho \geq 1$ is the constant from \eqref{qdr}. To be more specific, we have that  for all $g'\in B_\rho(\tilde g, 4 C_\rho^2 \varepsilon_{\tilde{g}})$, the values
$K_j(g')$ and $K_j(\tilde g)$ have the same sign.

Since $\Omega$ is compact, and 
$$\bigcup_{\tilde g\in \Omega}B_\rho\big(\tilde{g}, \varepsilon_{\tilde{g}}\big)\supset \Omega,$$
we have a finite subcover, say $B_\rho(\tilde{g}_1, \varepsilon_{\tilde{g}_1}),\cdots, B_\rho(\tilde{g}_m, \varepsilon_{\tilde{g}_m})$. 
Then for any $\tilde g\in\Omega$, there exists $1\leq l\leq m$ such that
$$\tilde g\in B_\rho\big(\tilde{g}_l, \varepsilon_{\tilde{g}_l}\big).$$

For any fixed $g\in \mathcal G$, let
\begin{align*}
 g_*= g \circ \delta_r( \tilde g^{-1} ).
\end{align*}
Then 
\begin{align*}
\rho(g,g_*) = \rho(g, g \circ \delta_r( \tilde g^{-1} )) =r.
\end{align*}

Let $\varepsilon_o=\min_{1\leq l\leq m}\{\varepsilon_{\tilde{g}_l}\}$, for every $\eta\in(0,  \varepsilon_o)$,
 we consider the two balls $B_\rho( g,  \eta r)$ and $B_\rho( g_*,  \eta r) $.
 It is clear that for every $g_1\in B_\rho( g,  \eta r)$, we can write
 $$ g_1 = g\circ \delta_r (g'_1),$$
where $g'_1 \in B_\rho(0,  \eta)$.
Similarly, for every $g_2\in B( g_*,  \eta r)$, we can write
 $$ g_2 = g_*\circ \delta_r (g'_2), $$
where $g'_2 \in B_\rho(0,  \eta)$.

As a consequence, we have
\begin{align}\label{dilation}
K_j(g_1,g_2) &= K_j\big(   g\circ \delta_r (g'_1) ,   g_*\circ \delta_r (g'_2)   \big)\\
&= K_j\big(   g\circ \delta_r (g'_1) ,    g \circ \delta_r( \tilde g^{-1} )\circ \delta_r (g'_2)   \big)\nonumber\\
&= r^{-Q} K_j\big(    (g'_2)^{-1} \circ    \tilde g\circ g'_1  \big).\nonumber
\end{align}
Similarly,
\begin{align}\label{dilation0}
K_j(g_2,g_1) = r^{-Q} K_j\big(    (g'_1)^{-1} \circ    \tilde g\circ g'_2  \big).
\end{align}

Next, we note that
\begin{align*}
\rho\big(    (g'_2)^{-1} \circ    \tilde g \circ g'_1, \tilde g_l \big) &= \rho\big(      \tilde g\circ g'_1, g'_2 \circ  \tilde g_l \big)\\
&\leq C_\rho^2 \, \left[ \rho\big(      \tilde g \circ g'_1,   \tilde g \big)+\rho\big(\tilde g, \tilde g_l\big)+ \rho\big(      \tilde g_l , g'_2 \circ  \tilde g_l \big) \right]\\
&\leq3 C_\rho^2 \varepsilon_{\tilde g_l},
\end{align*}
and also
$$\rho\big(    (g'_1)^{-1} \circ    \tilde g \circ g'_2, \tilde g_l \big) \leq3 C_\rho^2 \varepsilon_{\tilde g_l},$$
which shows that $ (g'_2)^{-1} \circ    \tilde g \circ g'_1$ and  $ (g'_1)^{-1} \circ    \tilde g \circ g'_2$ are contained in the ball $B_\rho(\tilde g_l, 4C_\rho^2 \varepsilon_{\tilde g_l})$
for all $g'_1 \in B_\rho(0,  \eta)$ and for all $g'_2 \in B_\rho(0,  \eta)$.

Thus, from  \eqref{non zero e1}, we obtain that
\begin{align}\label{lower bound e1}
| K_j\big(    (g'_2)^{-1} \circ    \tilde g \circ g'_1  \big)| > {1\over 2 } | K_j(\tilde g_l)|, \quad
| K_j\big(    (g'_1)^{-1} \circ    \tilde g \circ g'_2  \big)| > {1\over 2 } | K_j(\tilde g_l)|,
\end{align}
for all $g'_1 \in B_\rho(0,  \eta)$ and for all $g'_2 \in B_\rho(0,  \eta)$. Moreover, 
$K_j(    (g'_2)^{-1} \circ    \tilde g \circ g'_1 )$,  $K_j(    (g'_1)^{-1} \circ    \tilde g \circ g'_2  )$ and $K_j(\tilde g_l)$ have the same sign.

Now
combining  \eqref{dilation}, \eqref {dilation0} and \eqref{lower bound e1},  we obtain that
\begin{align}\label{lower bound e2}
|K_j(g_1,g_2)|   > {1\over 2 } r^{-Q} |K_j(\tilde g_l)|,\quad |K_j(g_2,g_1)|   > {1\over 2 } r^{-Q} |K_j(\tilde g_l)|
\end{align}
for every $g_1\in B_\rho( g,  \eta r)$ and for every $g_2\in B_\rho( g_*,  \eta r)$ , where $K_j(g_1,g_2)$, $K_j(g_2,g_1)$ and $K_j(\tilde g_l)$ have the same sign. Here $K_j(\tilde g_l)$ is a fixed
constant independent of $\eta$, $r$, $g$, $g_1$ and $g_2$. Set
$$C(K_j)={1\over 2}\min_{1\leq l\leq m}\left\{|K_j(\tilde g_l)|\right\}.$$

From the lower bound \eqref{lower bound e2} above, we further obtain that
for every $\eta\in (0,\varepsilon_o)$,
\begin{align*}
|K_j(g_1,g_2)|   > C(K_j) r^{-Q}, \quad |K_j(g_2,g_1)|> C(K_j) r^{-Q}
\end{align*}
for every $g_1\in B_\rho( g,  \eta r)$, every $g_2\in B_\rho( g \circ \delta_r( \tilde g^{-1} ),  \eta r)$ and every $\tilde g\in\Omega$. Moreover, since $K_j(\tilde g_l), 1\leq l\leq m,$ have the same sign, we can see, for any $\tilde g\in\Omega$,
the sign of $K_j(g_1,g_2)$ and of $K_j(g_2,g_1)$ are invariant, respectively, 
for every $g_1\in B_\rho( g,  \eta r)$ and every $g_2\in B_\rho(g \circ \delta_r( \tilde g^{-1} ),  \eta r)$.
\color{black}

It can be checked that there exists $r_*=r_*(s)>{10\over \varepsilon_o}$ such that for $r>r_*$, we have 
$$\max_{1\leq \nu\leq s }\big\{ (r^{\nu}-1)^{1\over \nu}\big\}=(r^{s}-1)^{1\over s}\quad {\rm{and}}\quad
\min_{1\leq \nu\leq s }\big\{ (r^{\nu}+1)^{1\over \nu}\big\}=(r^{s}+1)^{1\over s}.$$

Step 1, take $r_1=r_*$, we can have $\eta_1<\varepsilon_o$ such that $\eta_1r_1=1$. Let
$$E_1:=\big\{g'\in\tau_{g}\big(\delta_{\varrho}(\tilde\Omega)\big): (r_1^{s}-1)^{1\over s}<\varrho< (r_1^{s}+1)^{1\over s}\big\},$$
where $\tilde \Omega=\{g^{-1}: g\in\Omega\}$ and $\tau_{g}(\delta_{\varrho}(\tilde\Omega))=\{g\circ\delta_{\varrho}(\tilde g^{-1}): \tilde g\in \Omega\}.$ 
Recall that for any $g\in\mathcal G$, $g^{-1}=-g$, then we have
$$|E_1|
=\frac{\sigma(\Omega)}{Q}\big[(r_1^{s}+1)^{Q\over s}-(r_1^{s}-1)^{Q\over s}\big].$$

For any $g'\in E_1$, there exists $\tilde g=(x^{(1)},\cdots, x^{(s)})\in\Omega$ such that $g'=g\circ \delta_{\varrho}(\tilde g^{-1})$, where $x^{(\nu)}\in\mathbb{R}^{n_{\nu}}, \nu=1,\cdots, s$. Then by Lemma \ref{gp law} and \eqref{hom norm}, we have
\begin{align*}
\rho(g', g\circ\delta_{r_1}(\tilde g^{-1}))&=\rho\big(g\circ \delta_{\varrho}(\tilde g^{-1}), g\circ\delta_{r_1}(\tilde g^{-1})\big)
=\rho\big(\delta_{\varrho}(\tilde g^{-1}), \delta_{r_1}(\tilde g^{-1})\big)\\
&=\Big(\sum_{\nu=1}^{s}\big|\left(\varrho^{\nu}-r_1^{\nu}\right)(-x)^{(\nu)}\big|^{{2s!}\over \nu}\Big)^{1\over {2s!}}\\
&\leq\max_{1\leq\nu\leq s}|\varrho^{\nu}-r_1^{\nu}|^{1\over \nu}<1,
\end{align*}
which implies that $g'\in B_\rho(g\circ\delta_{r_1}(\tilde g^{-1}),1)$. Therefore, by our claim, for any $g_1\in B_\rho(g,1)$, 
$$|K_j(g_1,g')|\geq C(K_j)r_1^{-Q}\geq C(K_j,Q) \rho(g_1,g')^{-Q},$$
and also
$$|K_j(g',g_1)|\geq C(K_j,Q) \rho(g_1,g')^{-Q}.$$
Moreover,  for every $g'\in E_1$ and every $g_1\in B_\rho(g,1)$,  all $K_j(g_1,g')$ and all $K_j(g_1,g')$  have the same sign.

Step 2, take $r_2=(r_{1}^{s}+2)^{1\over s}$, we can choose $\eta_2<\varepsilon_o$ such that $\eta_2r_2=1$.
Let 
$$E_2:=\big\{g'\in\tau_{g}\big(\delta_{\varrho}(\tilde\Omega)\big): (r_2^{s}-1)^{1\over s}<\varrho< (r_2^{s}+1)^{1\over s}\big\}.$$
Then
$$ |E_2|=\frac{\sigma(\Omega)}{Q}\big[(r_2^{s}+1)^{Q\over s}-(r_2^{s}-1)^{Q\over s}\big].$$
Moreover, $E_2\cap E_1=\emptyset.$ By the same discussion as above, for any $g'\in E_2$ and any $g_1\in B_\rho(g,1)$, we have 
$$|K_j(g_1,g')|\geq  C(K_j,Q) \rho(g_1,g')^{-Q},\quad |K_j(g', g_1)|\geq  C(K_j,Q) \rho(g_1,g')^{-Q},$$
and all $K_j(g_1,g')$ and all $K_j(g',g_1)$ have the same sign as those when $g'\in E_1$.

In general, for $l\geq 2$, take $r_l=(r_{l-1}^{s}+2)^{1\over s}$ and let
$$E_l:=\big\{g'\in\tau_{g}\big(\delta_{\varrho}(\tilde\Omega)\big): (r_l^{s}-1)^{1\over s}<\varrho< (r_l^{s}+1)^{1\over s}\big\}.$$
Then
$$ |E_l|=\frac{\sigma(\Omega)}{Q}\big[(r_l^{s}+1)^{Q\over s}-(r_l^{s}-1)^{Q\over s}\big].$$
Moreover, $E_l\cap E_{l-1}=\emptyset.$ By the same discussion as above, for any $g'\in E_l$ and any $g_1\in B(g,1)$, we have 
$$|K_j(g_1,g')|\geq  C(K_j,Q) \rho(g_1,g')^{-Q},\quad |K_j(g,g_1)|\geq  C(K_j,Q) \rho(g_1,g')^{-Q},$$
and all $K_j(g_1,g')$ as well as all $K_j(g',g_1)$ have the same sign as those when $g'\in \cup_{\nu=1}^{l-1}E_{\nu}$.

Set
$$G_g=\bigcup_{l=1}^{\infty}E_l,$$
and $r_o:=(r_*^s-1)^{s}$,
 then
  $\inf_{g'\in G}\rho(g,g')=r_o$, and for every $g_1\in B_\rho(g,1)$ and $g_2\in G_g$, we have 
  $$|K_j(g_1,g_2)|\geq  C(K_j,Q) \rho(g_1,g_2)^{-Q},
  \quad K_j(g_2,g_1)|\geq  C(K_j,Q) \rho(g_1,g_2)^{-Q},$$
  and all $K_j(g_1,g_2)$, $K_j(g_2,g_1)$ have the same sign.
  Moreover, $|G_g|=\infty$ and for any $R_2>R_1>2r_o$,
  $$\left|\left(B_\rho(g,R_2)\setminus B_\rho(g,R_1)\right)\cap G_g\right|=\frac{\sigma(\Omega)}{Q}\big(R_2^{Q}-R_1^{Q}\big)\approx \left|B_\rho(g,R_2)\setminus B_\rho(g,R_1)\right|,$$
where the implicit constants are independent of $g, R_1$ and $R_2$.
This completes the proof of Theorem \ref{thm0}.
\end{proof}

By performing minor modification in the above proof, 
we can also get the similar result for any ball $B(g, R_0)$.
 \begin{corollary}\label{cor1}
Suppose that $\mathcal G$ is a stratified nilpotent Lie group with homogeneous dimension $Q$ and that $j\in\{1,2,\ldots,n\}$. Now let $C_{d_2}$ be the constant appeared in   \eqref{equi metric d}.
There exist a large positive constant $r_o$ and a positive constant $C$ depending on $K_j$, $Q$ and $C_{d_2}$ such that for every 
$g\in \mathcal G$ there exists a set $G_g\subset \mathcal G$ such that  $\inf_{g'\in G_g}\rho(g,g')=r_oR_0$ and that  for every $g_1\in B_\rho(g,R_0)$ and $g_2\in G_g$, we have
$$ |K_j(g_1,g_2)|\geq C d(g_1,g_2)^{-Q}, \quad
  |K_j(g_2,g_1)|\geq C  d(g_1,g_2)^{-Q},$$
all $K_j(g_1,g_2)$ as well as all $K_j(g_2,g_1)$ have the same sign.  

Moreover, the set $G_g$ is regular, in the sense that $|G_g|=\infty$ and that for any $R_2>R_1>2r_oR_0$,
$$ |\left(B_\rho(g,R_2)\setminus B_\rho(g,R_1) \right) \cap G_g|\approx \left|B_\rho(g,R_2)\setminus B_\rho(g,R_1)\right|,  $$
where the implicit constants are independent of $g, R_1$ and $R_2$. 
\end{corollary}

\begin{proof}
By  taking $r_{1}=r_* R_0$, $r_j=(r_{j-1}^{s}+2R_0^s)^{1\over s}$ for $j\in\mathbb N$ and $j\geq 2$, and
$$E_l:=\left\{g'\in\tau_{g}\left(\delta_{\varrho}(\tilde\Omega)\right): (r_l^{s}-R_0^s)^{1\over s}<\varrho< (r_l^{s}+R_0^s)^{1\over s}\right\},\quad l\in\mathbb N,$$
in the proof of Theorem \ref{thm0}, we can see that
for any $g'\in E_l$ and any $g_1\in B(g,R_0)$, we have 
\begin{align*}
|K_j(g_1,g')|&\geq  C(K_j,Q) \rho(g_1,g')^{-Q}\geq  C(K_j,Q,C_{d_2}) d(g_1,g')^{-Q},\\
 |K_j(g,g_1)|&\geq  C(K_j,Q) \rho(g_1,g')^{-Q}\geq  C(K_j,Q,C_{d_2}) d(g_1,g')^{-Q},
 \end{align*}
and all $K_j(g_1,g')$ as well as all $K_j(g',g_1)$ have the same sign as those when $g'\in \cup_{\nu=1}^{l-1}E_{\nu}$. Set
$$G_g=\bigcup_{l=1}^{\infty}E_l,$$
 then
  $\inf_{g'\in G}\rho(g,g')=r_oR_o$, and for every $g_1\in B_\rho(g,R_0)$ and $g_2\in G_g$, we have 
  $$|K_j(g_1,g_2)|\geq  C(K_j,Q,C_{d_2}) d(g_1,g_2)^{-Q},
  \quad K_j(g_2,g_1)|\geq  C(K_j,Q,C_{d_2}) d(g_1,g_2)^{-Q},$$
  and all $K_j(g_1,g_2)$, $K_j(g_2,g_1)$ have the same sign. The rest part of the proof is the same as that of Theorem \ref{thm0}.
\end{proof}

\section{Compactness of Riesz transform commutator}

In this section, we will give the proof of Theorem \ref{thm1}.  
 We need the following upper and lower bounds for integrals of $[b, \mathcal R_{\ell}]$.
\begin{lemma}\label{lem up lw}
Assume that $b\in {\rm BMO}(\mathcal G)$ with $\|b\|_{{\rm BMO}(\mathcal G)}=1$ and there exist $\delta>0$ and a sequence $\{B_j\}_{j=1}^{\infty}:=\{B(g_j, r_j)\}_{j=1}^{\infty}$ of balls such that for each $j$,
\begin{align}
M(b,B_j)>\delta.
\end{align}
Then there exist functions $\{f_j\}\subset L_w^p(\mathcal G)$ with $\|f_j\|_{L_w^p(\mathcal G)}=1$, positive constants  $\beta_1>{C_{d_2}\over C_{d_1}}r_o, \beta_2, \beta_3$ such that for any integers $k\geq [\log_2\beta_1]$ and $j$,
\begin{align}\label{upper}
\int_{(2^{k+1}B_j\setminus 2^k B_j)\cap G_{g_j}}\big|[b,\mathcal R_{\ell}]f_j(g)\big|^p w(g)dg\geq \beta_2\delta^p 2^{-Qkp}{w(2^{k+1}B_j)\over w(B_j)},
\end{align} 
and
\begin{align}\label{lower}
\int_{2^{k+1}B_j\setminus 2^k B_j}\big|[b,\mathcal R_{\ell}]f_j(g)\big|^p w(g)dg\leq \beta_3 2^{-Qkp}{w(2^{k+1}B_j)\over w(B_j)},
\end{align} 
where $C_{d_1}$ and $C_{d_2}$ are in \eqref {equi metric d},  $r_o$ and $G_{g_j}$ are the same as those in Theorem \ref{thm0}.
\end{lemma}

\begin{proof}
For every $j\in\mathbb N $, we define $f_j$ as follows. By the definition of median value, we can find disjoint subsets $E_{j1}, E_{j2}\subset B_j$ such that 
\begin{align*}
E_{j1}\supset\{g\in B_j: b(g)\geq m_b(B_j)\},\quad E_{j2}\supset\{g\in B_j: b(g)\leq m_b(B_j)\},
\end{align*}
and
$|E_{j1}|=|E_{j2}|={1\over 2} \left|B_j\right|.$
Define 
$f_j(g)=w(B_j)^{-{1\over p}}\left(\chi_{E_{j1}}(g)-\chi_{E_{j2}}(g)\right).$
Then $f_j$ satisfies $\supp f_j\subset B_j$ and for every $g\in B_j$, 
\begin{align}\label{fjpro}
|f_j(g)|=w(B_j)^{-{1\over p}}, \quad f_j(g)\left(b(g)-m_b(B_j)\right)\geq0.
\end{align}
Moreover, 
\begin{align}\label{fjpro1}
\int_{B_j} f_j(g)dg=0,\quad 
\|f_j\|_{L^p_w(\mathcal G)}=1.
\end{align}

Note that 
$[b,\mathcal R_{\ell}]f=\mathcal R_{\ell}\big( (b-m_b(B_j))f\big)-\big(b-m_b(B_j)\big)\mathcal R_{\ell}(f).$
For any $g\in\mathcal G\setminus (2B_j)$, by \eqref {kj} and \eqref {fjpro1}, we have
\begin{align*}
\left| \big(b-m_b(B_j)\big)\mathcal R_\ell(f_j) (g)\right|&=\left| b(g)-m_b(B_j)\right|\left|\int_{B_j}\left(K_\ell(g,g')-K_\ell(g,g_j) \right)f_j(g')dg'  \right|\\
&\leq\left| b(g)-m_b(B_j)\right|\int_{B_j}\left|K_\ell(g,g')-K_\ell(g,g_j) \right| \left|f_j(g')\right|dg'  \\
&\lesssim\left| b(g)-m_b(B_j)\right|\int_{B_j}{d(g',g_j)\over d(g,g_j)^{Q+1}} \left|f_j(g')\right|dg'\\
&\lesssim {r_j\over d(g,g_j)^{Q+1}}\left(w(B_j)\right)^{-{1\over p}} |B_j| \left| b(g)-m_b(B_j)\right|.
\end{align*}

By John-Nirenberg inequality  (c.f. \cite{CRT}), for each $l\in\mathbb N$ and $B\subset\mathcal G$,
\begin{align}\label{bkp}
\nonumber\int_{2^{l+1}B} \left| b(g)-m_b(B)\right|^pdg
&\lesssim\int_{2^{l+1}B} \big| b(g)-m_b(2^{l+1}B)\big|^pdg+\big|2^{l+1}B\big| \big|m_b(2^{l+1}B)-m_b(B) \big|^p\\
&\lesssim l^p \big| 2^{l+1}B\big|.
\end{align}

Since $w\in A_p(\mathcal G)$, there exists $r>1$ such that $w\in RH_r(\mathcal G)$. Then by H\"older's inequality and \eqref{bkp}, we have
\begin{align} \label{bmup}
& \nonumber\int_{2^{k+1}B_j\setminus 2^{k}B_j} \left| \big(b-m_b(B_j)\big)\mathcal R_\ell(f_j) (g)\right|^pw(g)dg\\
& \nonumber\lesssim {r_j^p|B_j|^p\over w(B_j)}\int_{2^{k+1}B_j\setminus 2^{k}B_j}{1\over d(g,g_j)^{p(Q+1)}}\left| b(g)-m_b(B_j)\right|^pw(g)dg\\
&\lesssim {1\over 2^{kp(Q+1)}w(B_j)}\int_{2^{k+1}B_j\setminus 2^{k}B_j}\left| b(g)-m_b(B_j)\right|^pw(g)dg\\
& \nonumber\lesssim {1\over 2^{kp(Q+1)}w(B_j)}\left(\int_{2^{k+1}B_j\setminus 2^{k}B_j}\left| b(g)-m_b(B_j)\right|^{pr'}dg\right)^{1\over r'}
\left(\int_{2^{k+1}B_j}w(g)^{r}dg\right)^{1\over r}\\
& \nonumber\lesssim  {k^p\over 2^{kp(Q+1)}w(B_j)}\left| 2^{k+1}B_j\right|^{1\over r'}.\left| 2^{k+1}B_j\right|^{1\over r}
\left({1\over | 2^{k+1}B_j|}\int_{2^{k+1}B_j}w(g)^{r}dg\right)^{1\over r}\\
& \nonumber\leq \beta_4{k^p\over 2^{kp(Q+1)}}{w(2^{k+1}B_j)\over w(B_j)}.
\end{align}

For $g\in(\mathcal G\setminus 2^kB_j)\cap G_{g_j}$, $k>[\log_2({C_{d_2}\over C_{d_1}}r_o)]$, by Corollary \ref{cor1} and \eqref{kj}, we have 
\begin{align*}
\left| \mathcal R_{\ell}\big( (b-m_b(B_j))f_j(g)\big)\right|&=\int_{E_{j1}\cup E_{j2}}\left|K_\ell(g,g')\right|\big(b(g')-m_b(B_j)\big)f_j(g')dg' \\
&\gtrsim\int_{E_{j1}\cup E_{j2}}{1\over d(g, g')^Q}{\big|b(g')-m_b(B_j)\big|\over w(B_j)^{{1\over p}}}dg' \\
&\gtrsim {1\over w(B_j)^{1\over p} d(g, g_j)^Q}\int_{B_j}\big|b(g')-m_b(B_j)\big|dg'\\
&\gtrsim {\delta |B_j|\over w(B_j)^{1\over p} d(g, g_j)^Q}.
\end{align*}
Using Corollary \ref{cor1} again and Lemma \ref{lemlw}, we have
\begin{align}\label{bmlw}
&\nonumber\int_{(2^{k+1}B_j\setminus 2^{k}B_j)\cap G_{g_j}}\left| \mathcal R_{\ell}\big( \big(b-m_b(B_j)\big)f_j(g)\big)\right|^pw(g)dg\\
&\nonumber\gtrsim {\delta^p |B_j|^p\over w(B_j)}\int_{(2^{k+1}B_j\setminus 2^{k}B_j)\cap G_{g_j}}{1\over d(g, g_j)^{Qp}}w(g)dg\\
&\gtrsim {\delta^p \over 2^{Qkp}w(B_j)}\int_{(2^{k+1}B_j\setminus 2^{k}B_j)\cap G_{g_j}}w(g)dg\\
&\nonumber\gtrsim {\delta^p \over 2^{Qkp}w(B_j)}\left({|(2^{k+1}B_j\setminus 2^{k}B_j)\cap G_{g_j}|\over |2^{k+1}B_j|} \right)^p w\big(2^{k+1}B_j \big)\\
&\nonumber\gtrsim {\delta^p w(2^{k+1}B_j ) \over 2^{Qkp}w(B_j)}
\geq \beta_2{\delta^p\over 2^{Qkp}}{w(2^{k+1}B_j)\over w(B_j)}.
\end{align}

Take $\beta_1> {C_{d_2}\over C_{d_1}} r_o$ large enough such that for any integer $k\geq[\log_2 \beta_1]$, 
$$\beta_2 {\delta^p\over 2^{kpQ+p-1}}-\beta_4{k^p\over 2^{kp(Q+1)}}\geq \beta_2 {\delta^p\over 2^{kpQ+p}}.$$
Then by  \eqref{bmup} and \eqref{bmlw}, for any $k\geq[\log_2 \beta_1]$, we have
\begin{align*}
&\int_{(2^{k+1}B_j\setminus 2^{k}B_j)\cap G_{g_j}}\left| [b,\mathcal R_{\ell}]f_j(g)\big)\right|^pw(g)dg\\
&\geq{1\over 2^{p-1}}\int_{(2^{k+1}B_j\setminus 2^{k}B_j)\cap G_{g_j}}\left| \mathcal R_{\ell}\big( \big(b-m_b(B_j)\big)f_j(g)\big)\right|^pw(g)dg\\
&\quad-\int_{(2^{k+1}B_j\setminus 2^{k}B_j)} \left| \big(b-m_b(B_j)\big)R_l(f_j) (g)\right|^pw(g)dg\\
&\geq \left(\beta_2 {\delta^p\over 2^{kpQ+p-1}}-\beta_4{k^p\over 2^{kp(Q+1)}}\right)
{w(2^{k+1}B_j)\over w(B_j)}\\
&\geq \beta_2 {\delta^p\over 2^{kpQ+p}}{w(2^{k+1}B_j)\over w(B_j)}.
\end{align*}

On the other hand,  for $g\in\mathcal G\setminus 2B_j$, we have
\begin{align*}
\left| \mathcal R_{\ell}\big( \big(b-m_b(B_j)\big)f_j(g)\big)\right|&\leq {1\over w(B_j)^{1\over p}}\int_{B_j}\left|K_\ell(g,g')\right|\big|b(g')-m_b(B_j)\big|dg' \\
&\lesssim  {1\over w(B_j)^{1\over p}}\int_{B_j}{1\over d(g,g')^Q}\big|b(g')-m_b(B_j)\big|dg' \\
&\lesssim {1\over w(B_j)^{1\over p}d(g,g_j)^Q}\int_{B_j}\big|b(g')-m_b(B_j)\big|dg' \\
&\lesssim {1\over w(B_j)^{1\over p}d(g,g_j)^Q}|B_j|.
\end{align*}
Therefore,
\begin{align*}
\int_{2^{k+1}B_j\setminus 2^{k}B_j}\left| \mathcal R_{\ell}\big( \big(b-m_b(B_j)\big)f_j(g)\big)\right|^pw(g)dg
&\leq{|B_j|^p\over w(B_j)}\int_{2^{k+1}B_j\setminus 2^{k}B_j}{1\over d(g,g_j)^{Qp}}w(g)dg\\
&\leq {|B_j|^p\over w(B_j)}{1\over (2^{k}r_j)^{pQ}}w\big( 2^{k+1}B_j\setminus 2^{k}B_j\big)\\
&\lesssim{1\over 2^{kpQ}} {w(2^{k+1}B_j)\over w(B_j)}.
\end{align*}
Take $k$ large enough such that ${k\over 2^k}<1$, then by \eqref{bmup}, we can obtain
\begin{align*}
\int_{2^{k+1}B_j\setminus 2^{k}B_j}\left| [b,\mathcal R_{\ell}]f_j(g)\big)\right|^pw(g)dg
&\leq\int_{2^{k+1}B_j\setminus 2^{k}B_j}\left| \mathcal R_{\ell}\big( \big(b-m_b(B_j)\big)f_j(g)\big)\right|^pw(g)dg\\
&\quad+\int_{2^{k+1}B_j\setminus 2^{k}B_j} \left| \big(b-m_b(B_j)\big)R_l(f_j) (g)\right|^pw(g)dg\\
&\lesssim \left({1\over 2^{kpQ}}+{k^p\over 2^{kp(Q+1)}}\right)
{w(2^{k+1}B_j)\over w(B_j)}\\
&\leq \beta_3{1\over 2^{kpQ}}{w(2^{k+1}B_j)\over w(B_j)}.
\end{align*}
This completes the proof of Lemma \ref{lem up lw}.
\end{proof}

Recall that in \cite{DLLWW}, we have established  the Bloom-type two weight estimates for the commutators $[b,\mathcal R_{j}]$. From this result, for the case of one weight, we have the following estimates.
\begin{lemma}[\cite{DLLWW}, Theorem 1.2]\label{lem one weight}
Suppose $w\in A_p(\mathcal G)$ and $j\in\{1,\ldots,n\}$. Then

{\rm (i)} if $b\in {\rm BMO}(\mathcal G)$, then 
$$ \|[b,\mathcal R_j](f)\|_{L^p_w(\mathcal G)} \lesssim \|b\|_{  {\rm BMO}(\mathcal G) } \|f\|_{L^p_w(\mathcal G)}.$$

{\rm (ii)} for every $b\in L^1_{loc}(\mathcal G)$, if 
$[b, \mathcal R_j]$ is bounded on $L^p_w(\mathcal G)$, then
$b\in {\rm BMO}(\mathcal G)$ with 
$$     
\|b\|_{  {\rm BMO}(\mathcal G) } \lesssim   \|[b,\mathcal R_j]\|_{L^p_w(\mathcal G)\to L^p_w(\mathcal G)}. $$
\end{lemma}

G\'orka and Macios  established the Riesz-Kolmogorov theorem on doubling measure spaces \cite[Theorem 1]{GM}.
Since $A_p$ weights are doubling, we have the following corresponding result.
\begin{lemma}\label{lem r-k}
Let $1<p<\infty$, $g_0\in\mathcal G$. Then the subset $\mathcal F$ of $L^p_w(\mathcal G)$ is relatively
compact in $L^p_w(\mathcal G)$ if and only if the following conditions are satisfied:
 
 {\rm (i)}  $\mathcal F$ is bounded.
 
  {\rm (ii)}  $$\lim_{R\rightarrow\infty}\int_{\mathcal G\setminus B(g_0, R)}|f(g)|^p w(g)dg=0$$ uniformly for $f\in\mathcal F$.
  
  {\rm (iii)}  $$\lim_{r\rightarrow 0}\int_{\mathcal G}|f(g)-f_{B(g,r)}|^p w(g)dg=0$$ uniformly for $f\in\mathcal F$.
\end{lemma}

For the proof of  Theorem \ref{thm1}, we also need to establish the characterisation of ${\rm VMO}(\mathcal G)$. We will give its proof in Appendix. For the Euclidean case one can refer to \cite{U1}. 
\begin{theorem}\label{thm-cmo}
Let $f\in {\rm BMO}(\mathcal G)$. Then $f\in{\rm VMO}(\mathcal G)$ if and only if $f$ satisfies the following three conditions.
\begin{flalign*}\begin{split}
&{\rm (i)} \quad \quad\lim\limits_{a\rightarrow 0}\sup\limits_{r_B=a}M(f,B)=0.\\
&{\rm (ii)}  \quad\ \ \lim\limits_{a\rightarrow \infty}\sup\limits_{r_B=a}M(f,B)=0.\\
&{\rm (iii)}  \quad\ \lim\limits_{r\rightarrow \infty}\sup\limits_{B\subset \mathcal G\setminus B(0,r)}M(f,B)=0.
\end{split}&
\end{flalign*}
\end{theorem}

Now we prove Theorem \ref{thm1}.

\begin{proof}[Proof of Theorem \ref{thm1}]
{\bf Sufficient condition:}
Assume that $[b, \mathcal R_{\ell}]$ is compact on $L^p_w(\mathcal G)$, then $[b,\mathcal R_{\ell}]$ is bounded on $L^p_w(\mathcal G)$. By Lemma \ref{lem one weight}, we have $b\in {\rm BMO}(\mathcal G)$. Without loss of generality, we may assume that $\|b\|_{  {\rm BMO}(\mathcal G) }=1$. To show $b\in {\rm VMO}(\mathcal G)$, we may use a contradiction argument via Theorem \ref {thm-cmo}. Suppose that $b\notin {\rm VMO}(\mathcal G)$, then $b$ does not satisfy at least one of the three conditions in Theorem \ref{thm-cmo}. We will consider these three cases seperately.

{\bf Case (i)}. Suppose $b$ does not satisfy (i) in Theorem \ref{thm-cmo}. Then there exist $\delta>0$ and a sequence $\{B_j\}_{j=1}^\infty:=\{B(x_j, r_j)\}_{j=1}^\infty$ of balls such that $M(f, B_j)>\delta$ and $r_j\rightarrow 0$ as $j\rightarrow \infty$. Let $f_j$, $\beta_1>{C_{d_2}\over C_{d_1}}r_o, \beta_2, \beta_3$ be as in Lemma \ref{lem up lw} and $C_1, C_2$ in Lemma \ref{lemlw} and $\gamma_1>\beta_1$ large enough such that
\begin{align}\label{gamma2}
\gamma_2^p
:={\beta_2\over C_2}\delta^p 2^{Q(4-{4\over r}-3p)}\beta_1^{Q(1-p-{1\over r})}
>{\beta_3\over C_1} {2^{p+Q(p-\sigma-[\log_2\gamma_1]\sigma)}\over 1-2
^{-Q\sigma}},
\end{align}
where $\sigma$ is in \eqref{epsilon}.

Since $r_j\rightarrow 0$ as $j\rightarrow \infty$, we may choose a subsequence $\{B^{(1)}_{j_i}\}$ of $\{B_j\}$ such that
\begin{align}\label{ratio1}
{|B^{(1)}_{j_{i+1}}|\over |B^{(1)}_{j_{i}}|}\leq{1\over \gamma_1^Q}.
\end{align}

For fixed $i, m\in\mathbb N$, denote
\begin{align*}
\Omega:=\gamma_1B^{(1)}_{j_{i}}\setminus \beta_1B^{(1)}_{j_{i}},\quad \Omega_1:=\Omega \setminus \gamma_1B^{(1)}_{j_{i+m}},\quad
\Omega_2:=\mathcal G \setminus \gamma_1B^{(1)}_{j_{i+m}}.
\end{align*}
It is clear that 
$$\Omega_1\subset \gamma_1B^{(1)}_{j_{i}}\cap \Omega_2,\quad 
\Omega_1=\Omega\setminus \left(\Omega\setminus \Omega_2\right). $$
Then we have
\begin{align*}
&\left\| [b, \mathcal R_{\ell} ](f_{j_i})- [b, \mathcal R_{\ell} ](f_{j_{i+m}})\right\|_{L^p_w(\mathcal G)}\\
&\geq \left( \int_{\Omega_1}\left|  [b, \mathcal R_{\ell} ](f_{j_i})(g)- [b, \mathcal R_{\ell} ](f_{j_{i+m}})(g)\right|^p w(g) dg \right)^{1\over p}\\
&\geq \left( \int_{\Omega_1}\left|  [b, \mathcal R_{\ell} ](f_{j_i})(g)\right|^p w(g) dg \right)^{1\over p} -
\left( \int_{\Omega_2}\left|  [b, \mathcal R_{\ell} ](f_{j_{i+m}})(g)\right|^p w(g) dg \right)^{1\over p}\\
&=\left( \int_{\Omega\setminus \left(\Omega\setminus \Omega_2\right)}\left|  [b, \mathcal R_{\ell} ](f_{j_i})(g)\right|^p w(g) dg \right)^{1\over p} -
\left( \int_{\Omega_2}\left|  [b, \mathcal R_{\ell} ](f_{j_{i+m}})(g)\right|^p w(g) dg \right)^{1\over p}\\
&=:I_1-I_2.
\end{align*}

We first consider the term $I_1$. Assume that $\Gamma_{j_i}:=\Omega\setminus \Omega_2\neq\emptyset$, then 
$\Gamma_{j_i}\subset \gamma_1B^{(1)}_{j_{i+m}}$. Hence,  by \eqref{ratio1}, we have
$$|\Gamma_{j_i}|\leq \left|\gamma_1B^{(1)}_{j_{i+m}}\right|=\gamma_1^Q\left| B^{(1)}_{j_{i+m}}\right|\leq\left| B^{(1)}_{j_{i}}\right|.$$
Now for each $k\geq [\log_2 \beta_1]$,
$$\left|2^{k+1}B_{j_i}^{(1)}\setminus 2^k B_{j_i}^{(1)}\right|=|B(0,1)|\left( 2^{(k+1)Q}-2^{kQ}\right)r^Q_{j_i}
>\left| 2^k B_{j_i}^{(1)}\right|>|\Gamma_{j_i}|.$$
From this fact, it follows that there exist at most two rings, $2^{k_0+2}B_{j_i}^{(1)}\setminus 2^{k_0+1}B_{j_i}^{(1)}$ and $2^{k_0+1}B_{j_i}^{(1)}\setminus 2^{k_0}B_{j_i}^{(1)}$ such that $\Gamma_{j_i}\subset (2^{k_0+2}B_{j_i}^{(1)}\setminus 2^{k_0+1}B_{j_i}^{(1)})\cup (2^{k_0+1}B_{j_i}^{(1)}\setminus 2^{k_0}B_{j_i}^{(1)})$.
Then by \eqref{upper} and Lemma \ref{lemlw}, we have
\begin{align}\label{i1}
I_1^p&=\int_{\Omega\setminus \left(\Omega\setminus \Omega_2\right)}\big|  [b, \mathcal R_{\ell} ](f_{j_i})(g)\big|^p w(g) dg\\
&\nonumber\geq \sum_{k=[\log_2 \beta_1]+1, ~k\neq k_0, k_0+1}^{[\log_2 \gamma_1]-1}\int_{(2^{k+1}B^{(1)}_{j_i}\setminus 2^{k}B^{(1)}_{j_i})\cap G_{g_{j_i}}}\big|  [b, \mathcal R_{\ell} ](f_{j_i})(g)\big|^p w(g) dg\\
& \nonumber\geq\beta_2  \delta^p \sum_{k=[\log_2 \beta_1]+1, ~k\neq k_0, k_0+1}^{[\log_2 \gamma_1]-1} 2^{-kQp}{w(2^{k+1}B_{j_i}^{(1)})\over w(B_{j_i}^{(1)})}\\
& \nonumber\geq {\beta_2\over C_2} \delta^p \sum_{k=[\log_2 \beta_1]+1, ~k\neq k_0, k_0+1}^{[\log_2 \gamma_1]-1} 2^{-kQp}2^{(k+1)Q(1-{1\over r})}\\
& \nonumber\geq {\beta_2\over C_2}\delta^p 2^{Q(1-{1\over r})}\sum_{k=[\log_2 \beta_1]+3}^{[\log_2 \gamma_1]-1}2^{-kQ(p-1+{1\over r})}\\
&\nonumber \geq{\beta_2\over C_2}\delta^p 2^{Q(4-{4\over r}-3p)}\beta_1^{Q(1-p-{1\over r})}
=:\gamma_2^p.
\end{align}

If $\Omega\setminus\Omega_2=\emptyset$, the inequality above still holds.

For $I_2$, by \eqref{lower} in Lemma \ref{lem up lw}, we have
\begin{align*}
I_2^p&=\int_{\mathcal G \setminus \gamma_1B^{(1)}_{j_{i+m}}}\left|  [b, \mathcal R_{\ell} ](f_{j_{i+m}})(g)\right|^p w(g) dg\\
&\leq \sum_{k=[\log_2\gamma_1]}^{\infty}\int_{2^{k+1}B^{(1)}_{j_{i+m}}\setminus 2^{k}B^{(1)}_{j_{i+m}}}
\left|  [b, \mathcal R_{\ell} ](f_{j_{i+m}})(g)\right|^p w(g) dg\\
&\leq \beta_3 \sum_{k=[\log_2\gamma_1]}^{\infty}2^{-kQp}{w(2^{k+1}B^{(1)}_{j_{i+m}})\over w(B^{(1)}_{j_{i+m}})}
\end{align*}
By \cite[Theorem 1.2]{HPR}, for any $1<p<\infty$ and for every $w\in A_p(\mathcal G)$, there is an $\sigma=\sigma([w]_{A_p}, p, Q)$, with $0<\sigma<p$ such that 
$w\in A_{p-\sigma}(\mathcal G)$. Thus, by Lemma \ref{lemlw} and \eqref{gamma2}, we have
\begin{align}\label{epsilon}
I_2^p&\leq{\beta_3\over C_1} \sum_{k=[\log_2\gamma_1]}^{\infty}2^{-kQp}2^{(k+1)Q(p-\sigma)}
\leq {\beta_3\over C_1} {2^{Q(p-\sigma-[\log_2\gamma_1]\sigma)}\over 1-2
^{-Q\sigma}}<\left({\gamma_2\over 2}\right)^p.
\end{align}
Consequently, 
\begin{align*}
\left\| [b, \mathcal R_{\ell} ](f_{j_i})- [b, \mathcal R_{\ell} ](f_{j_{i+m}})\right\|_{L^p_w(\mathcal G)}\geq
{\gamma_2\over 2}.
\end{align*}
Therefore, $ \{[b, \mathcal R_{\ell} ](f_{j})\}_{j=1}^{\infty}$ is not relatively compact in $L^p_w(\mathcal G)$, which implies that $[b, \mathcal R_{\ell} ]$ is not compact on $L^p_w(\mathcal G)$. Thus $b$ satisfies condition (i).

{\bf Case (ii).} If $b$ does not satisfy (ii) in Theorem \ref{thm-cmo}, then there also exist $\delta>0$ and a sequence $\{B_j\}_{j=1}^\infty$ of balls such that $M(f, B_j)>\delta$ and $r_j\rightarrow \infty$ as $j\rightarrow \infty$. We take a subsequence $\{B_{j_i}^{(2)}\}$ of $B_j$ such that 
\begin{align}\label{B2}
{\big|B_{j_i}^{(2)}\big|\over \big|B_{j_{i+1}}^{(2)}\big|}\leq{1\over \gamma_1^Q}.
\end{align}
The method in this case is very similar to that in case (i), we just redefine our sets in a reversed order, i.e. for fixed $i$ and $m$, let
\begin{align*}
\tilde\Omega:=\gamma_1B^{(2)}_{j_{i+m}}\setminus \beta_1B^{(2)}_{j_{i+m}},\quad
\tilde\Omega_1:=\Omega \setminus \gamma_1B^{(2)}_{j_{i}} ,\quad
\tilde\Omega_2:=\mathcal G \setminus \gamma_1B^{(2)}_{j_{i}}.
\end{align*}
Then we have 
$$\tilde\Omega_1\subset \gamma_1B^{(2)}_{j_{i}}\cap \tilde\Omega_2,\quad 
\tilde\Omega_1=\tilde\Omega\setminus \left(\tilde\Omega\setminus \tilde\Omega_2\right). $$
Like in case (i), by Lemma \ref{lem up lw}  and \eqref{B2}, we can see that $[b, \mathcal R_{\ell} ]$ is not compact on $L^p_w(\mathcal G)$. Thus $b$ satisfies condition (ii) of Theorem \ref{thm-cmo}.

{\bf Case (iii).}  Assume that condition (iii) in Theorem \ref{thm-cmo} does not hold for $b$. Then there exists $\delta>0$ such that for any $r>0$, there exists $B\subset \mathcal G\setminus B(0,r)$ with $M(b,B)>\delta$.

We claim that for the $\delta$ above, there exists a sequence $\{B_j\}_j$ of balls such that for any $j$,
\begin{align}\label{mbb}
M(b, B_j)>\delta,
\end{align}
and for any $i\neq m$,
\begin{align}\label{bb}
\gamma_1B_i\cap \gamma_1B_m=\emptyset.
\end{align}
To see this, let $C_{\delta}>0$ to be determined later.  Then for $R_1>C_{\delta}$, there exists a ball $B_1:=B(g_1, r_1)\subset\mathcal G\setminus B(0,R_1)$ such that \eqref{mbb} holds. Similarly, for $R_j:=|g_{j-1}|+ 4\gamma_1C_{\delta}$, $j=2,3,\cdots,$ there exists $B_j:=B(g_j, r_j)\subset \mathcal G\setminus B(0, R_j)$ satisfying \eqref{mbb}. Repeating this procedure, we can obtain a sequence of balls $\{B_j\}_j$ with each $B_j$ satisfying \eqref{mbb}.
 Moreover, since $b$ satisfies condition (ii) in Theorem \ref{thm-cmo}, for the $\delta$ above, there exists a constant $\tilde C_\delta>0$ such that $M(b, B)<\delta$ for any ball $B$ satisfying $r_B>\tilde C_\delta$. This together with the choice of $\{B_j\}$ implies that $r_j\leq \tilde C_\delta=:C_\delta$ for all $j$. Therefore, for each $j$,
 $$\gamma_1r_j<\gamma_1C_\delta<4\gamma_1C_\delta.$$
Thus, for all $i\neq m$, without loss of generality, we may assume $i<m$, 
\begin{align*}
d(\gamma_1B_i, \gamma_1B_m)&\geq R_{i+1}-(|g_i|+\gamma_1r_i)-\gamma_1r_{i+1}
\geq 4\gamma_1C_\delta-2\gamma_1C_\delta=2\gamma_1C_\delta,
\end{align*}
which implies the claim.

We define
\begin{align*}
\hat \Omega_1:=\gamma_1B_{j}\setminus \beta_1B_{j},\quad
\hat \Omega_2:=\mathcal G\setminus \gamma_1B_{j+m}.
\end{align*}
Observe that $\hat \Omega_1\subset \hat \Omega_2$. Therefore, 
\begin{align*}
&\left\| [b, \mathcal R_{\ell} ](f_{j})- [b, \mathcal R_{\ell} ](f_{j+m})\right\|_{L^p_w(\mathcal G)}\\
&\geq \left( \int_{\hat\Omega_1}\left|  [b, \mathcal R_{\ell} ](f_j)(g)- [b, \mathcal R_{\ell} ](f_{j+m})(g)\right|^p w(g) dg \right)^{1\over p}\\
&\geq \left( \int_{\hat\Omega_1}\left|  [b, \mathcal R_{\ell} ](f_j)(g)\right|^p w(g) dg \right)^{1\over p} -
\left( \int_{\hat\Omega_2}\left|  [b, \mathcal R_{\ell} ](f_{j+m})(g)\right|^p w(g) dg \right)^{1\over p}\\
&=:\hat I_1-\hat I_2.
\end{align*}
By the similar estimates of $I_1$ and $I_2$ in case (i) and the definition of $\gamma_2$ in \eqref{i1}, we can deduce that $\hat I_1^p\geq \gamma_2^p$ and  $\hat I_2^p\leq ({\gamma_2\over 2})^p$. Consequently,
$$\left\| [b, \mathcal R_{\ell} ]f_{j}- [b, \mathcal R_{\ell} ]f_{j+m}\right\|_{L^p_w(\mathcal G)}\gtrsim \gamma_2,$$
which contadicts to the compactness of $[b, \mathcal R_\ell]$ on $L^p_w(\mathcal G)$, thereby $b$ also satisfies condition (iii) in Theorem \ref{thm-cmo}. This finishes the proof of the sufficiency of Theorem \ref{thm1}.

{\bf Necessary  condition:} Assume that $b\in {\rm VMO}(\mathcal G)$, we will show that $[b, \mathcal R_\ell]$ is compact on $L^p_w(\mathcal G)$.  
Since $b\in {\rm VMO}(\mathcal G)$, for any $\epsilon>0$, there exists $b_\epsilon\in  C_0^\infty(\mathcal G)$ such that 
$$\|b-b_\epsilon\|_{{\rm BMO}(\mathcal G)}<\epsilon.$$
By Lemma \ref {lem one weight}, we can see
\begin{align*}
\|[b,\mathcal R_\ell](f)-[b_\epsilon,\mathcal R_\ell](f)\|_{L^p_w(\mathcal G)}=
\|[b-b_\epsilon,\mathcal R_\ell](f)\|_{L^p_w(\mathcal G)}\lesssim  \|b-b_\epsilon\|_{{\rm BMO}(\mathcal G)}
\|f\|_{L^p_w(\mathcal G)}.
\end{align*}
Therefore,
$$\|[b,\mathcal R_\ell]-[b_\epsilon,\mathcal R_\ell]\|_{L^p_w(\mathcal G)\rightarrow L^p_w(\mathcal G)}
\lesssim  \|b-b_\epsilon\|_{{\rm BMO}(\mathcal G)}.$$
Thus, it suffices to show that $[b, \mathcal R_\ell]$ is a compact operator for $b\in C_0^\infty(\mathcal G)$.

Suppose $b\in C_0^\infty(\mathcal G)$, to show $[b, \mathcal R_\ell]$ is compact on $L^p_w(\mathcal G)$, it suffices to show that for every bounded subset $E\subset L^p_w(\mathcal G)$, the set $[b, \mathcal R_\ell] E$ is precompact. Thus, we only need to show that $[b, \mathcal R_\ell]E\subset L^p_w(\mathcal G)$ satisfies (i)-(iii) in Lemma \ref{lem r-k}. Firstly, by Lemma \ref{lem one weight}, $[b, \mathcal R_\ell] $ is bounded on $L^p_w(\mathcal G)$, which implies that $[b, \mathcal R_\ell] E$ satisfies (i) in Lemma \ref{lem r-k}.

Next we will show that $[b, \mathcal R_\ell] E$ satisfies (ii) in Lemma \ref{lem r-k}. We may assume that $b\in C_0^\infty(\mathcal G)$ with $\supp b\subset B(0,R)$.  Then for $t>2$, we have
\begin{align*}
&\left\|[b, \mathcal R_\ell]f\right\|_{L^p_w(\mathcal G\setminus B(0, tR))}\\
&=\left(\int_{\mathcal G\setminus B(0, tR)}\big|b(g) \mathcal R_\ell(f)(g) -\mathcal R_\ell(bf)(g)\big|^p w(g)dg\right)^{1\over p}\\
&\leq \left(\int_{\mathcal G\setminus B(0, tR)}\big|b \mathcal R_\ell(f)(g)\big|^p w(g)dg\right)^{1\over p}
+\left(\int_{\mathcal G\setminus B(0, tR)}\big|\mathcal R_\ell(bf)(g)\big|^p w(g)dg\right)^{1\over p}.
\end{align*}
Since $\supp b\subset B(0,R)$ and $B(0,R)\cap (\mathcal G\setminus B(0, tR))=\emptyset $,  the first term on the right hand side of the above inequality is zero. For $g\in \mathcal G\setminus B(0, tR)$, by \eqref{kj}, H\"older's inequality, the definition of $A_p$ weights and the fact that $b\in C_0^\infty(\mathcal G)$,  we have
\begin{align*}
\big|\mathcal R_\ell(bf)(g)\big|
&\leq \int_{\mathcal G}|K_\ell(g,g')|  |b(g')| |f(g')|dg'\lesssim \int_{B(0,R)}{1\over d(g,g')^Q}|b(g')| |f(g')|dg'\\
&\lesssim{1\over d(0,g)^Q}\int_{B(0,R)} |f(g')|w(g')^{1\over p} w(g')^{-{1\over p}}dg'\\
&\lesssim{1\over d(0,g)^Q}\|f\|_{L^p_w(\mathcal G)}
\left(\int_{B(0,R)} w(g')^{-{1\over p-1}}dg'\right)^{1-{1\over p}}\\
&\lesssim{1\over d(0,g)^Q}\|f\|_{L^p_w(\mathcal G)}{|B(0,R)|\over w(B(0,R))^{1\over p}}.
\end{align*}
Therefore,
\begin{align*}
&\left\|[b, \mathcal R_\ell]f\right\|_{L^p_w(\mathcal G\setminus B(0, tR))}
\lesssim \|f\|_{L^p_w(\mathcal G)}{|B(0,R)|\over w(B(0,R))^{1\over p}}
\left(\int_{\mathcal G\setminus B(0, tR)} {1\over d(0,g)^{Qp}}w(g)dg\right)^{1\over p}\\
&\lesssim \|f\|_{L^p_w(\mathcal G)}{|B(0,R)|\over w(B(0,R))^{1\over p}}
\left(\sum_{k=[\log_2 t]}^\infty\int_{2^{k+1}B(0,R)\setminus 2^kB(0,R)}{1\over d(0,g)^{Qp}}w(g)dg\right)^{1\over p}\\
&\lesssim \|f\|_{L^p_w(\mathcal G)}{|B(0,R)|\over w(B(0,R))^{1\over p}}
\left(\sum_{k=[\log_2 t]}^\infty{1\over (2^kR)^{Qp}}w\big(2^{k+1}B(0,R)\setminus 2^kB(0,R)\big)\right)^{1\over p}\\
&\lesssim \|f\|_{L^p_w(\mathcal G)}
\left(\sum_{k=[\log_2 t]}^\infty{1\over 2^{kQp}}{w\big(2^{k+1}B(0,R)\big)\over w(B(0,R))}\right)^{1\over p}\lesssim \|f\|_{L^p_w(\mathcal G)}\left(\sum_{k=[\log_2 t]}^\infty{1\over 2^{kQ\epsilon}}\right)^{1\over p}\\
&\lesssim \left({2^{Q\epsilon}\over 1-2^{-Q\epsilon}}\right)^{1\over p}\|f\|_{L^p_w(\mathcal G)} t^{-{Q\epsilon\over p}},
\end{align*}
which approches to zero as $t$ goes to infinity. This proves condition (ii) in Lemma \ref{lem r-k}.

At last, we prove that $[b, \mathcal R_\ell]E$ satisfies (iii) in Lemma \ref{lem r-k}. 
By a change of variables, we have
\begin{align*}
[b, \mathcal R_\ell]f(g)-\left([b,\mathcal R_\ell]f \right)_{B(g,r)}
&={1\over |B(g,r)|}\int_{B(g,r)} [b, \mathcal R_\ell]f(g)-[b, \mathcal R_\ell]f(g') dg'\\
&={1\over |B(0,r)|}\int_{B(0,r)} [b, \mathcal R_\ell]f(g)-[b, \mathcal R_\ell]f(\tilde gg) d\tilde g
\end{align*}

Take an arbitrary $\varepsilon\in(0,{1\over 2})$, $r\in\mathbb R_+$ and $\tilde g\in B(0,r)$. Then for any $g\in\mathcal G$,
\begin{align*}
& [b, \mathcal R_\ell]f(g)-[b, \mathcal R_\ell]f(\tilde gg)\\
&=\int_\mathcal G K_\ell(g,g') \left(b(g)-b(g')\right)f(g')dg'-\int_\mathcal G K_\ell(\tilde gg,g') \left(b(\tilde gg)-b(g')\right)f(g')dg'\\
&=\int_{d(g,g')>\varepsilon^{-1}d(\tilde g, 0)} K_\ell(g, g')\left(b(g)-b(\tilde g g)\right)f(g')dg'\\
&\quad +\int_{d(g,g')>\varepsilon^{-1}d(\tilde g, 0)}\left( K_\ell(g, g')-K_\ell(\tilde g g, g')\right)\left(b(\tilde gg)-b(g')\right)f(g')dg'\\
&\quad +\int_{d(g,g')\leq\varepsilon^{-1}d(\tilde g, 0)} K_\ell(g, g')\left(b(g)-b(g')\right)f(g')dg'\\
&\quad -\int_{d(g,g')\leq\varepsilon^{-1}d(\tilde g, 0)} K_\ell(\tilde gg, g')\left(b(\tilde g g)-b(g')\right)f(g')dg'\\
&=:L_1+L_2+L_3+L_4.
\end{align*}

Let us first consider the term $L_2$. By \eqref{kj}, we have
\begin{align*}
|L_2|&\leq \int_{d(g,g')>\varepsilon^{-1}d(\tilde g, 0)}\left| K_\ell(g, g')-K_\ell(\tilde g g, g')\right|\left|b(\tilde gg)-b(g')\right| |f(g')|dg' \\
&\lesssim d(\tilde g,0) \int_{d(g,g')>\varepsilon^{-1}d(\tilde g, 0)}{1\over d(g,g')^{Q+1}} |f(g')| dg'\\
&\lesssim d(\tilde g,0) \sum_{k=[\log_2\varepsilon^{-1}]}^\infty\int_{2^kd(\tilde g,0)<d(g,g')\leq 2^{k+1}d(\tilde g,0)}
{1\over d(g,g')^{Q+1}} |f(g')| dg'\\
&\lesssim d(\tilde g,0)\sum_{k=[\log_2\varepsilon^{-1}]}^\infty {1\over (2^kd(\tilde g, 0))^{Q+1}}\int_{d(g,g')\leq 2^{k+1}d(\tilde g,0)}|f(g')|dg'\\
&\lesssim \sum_{k=[\log_2\varepsilon^{-1}]}^\infty{1\over 2^k} M(f)(g)\\
&\lesssim \varepsilon M(f)(g),
\end{align*}
where $M(f)$ is the Hardy-Littlewood maximal operator on $\mathcal G$.

For $L_3$, by the mean value theorem and \eqref{kj}, we have
\begin{align*}
|L_3|&=\left|\int_{d(g,g')\leq\varepsilon^{-1}d(\tilde g, 0)} K_\ell(g, g')\left(b(g)-b(g')\right)f(g')dg'\right|\\
&\lesssim \int_{d(g,g')\leq\varepsilon^{-1}d(\tilde g, 0)}{1\over d(g,g')^{Q-1}}|f(g')|dg'\\
&\lesssim \sum_{k=-\infty}^{-1}\int_{2^k\varepsilon^{-1}d(\tilde g, 0)<d(g,g')\leq 2^{k+1}\varepsilon^{-1}d(\tilde g, 0)}
{1\over d(g,g')^{Q-1}}|f(g')|dg'\\
&\lesssim\sum_{k=-\infty}^{-1}{1\over (2^k\varepsilon^{-1}d(\tilde g, 0))^{Q-1}}\int_{d(g,g')\leq 2^{k+1}d(\tilde g, 0)}
|f(g')|dg'\\
&\lesssim \varepsilon^{-1}d(\tilde g, 0)M(f)(g)\sum_{k=-\infty}^{-1}2^k \lesssim\varepsilon^{-1}d(\tilde g, 0)M(f)(g).
\end{align*}

For $L_4$, again by \eqref{kj}, we can obtain
\begin{align*}
|L_4|&=\left| \int_{d(g,g')\leq\varepsilon^{-1}d(\tilde g, 0)} K_\ell(\tilde gg, g')\left(b(\tilde g g)-b(g')\right)f(g')dg' \right|\\
&\lesssim \int_{d(g,g')\leq\varepsilon^{-1}d(\tilde g, 0)} {d(\tilde gg,g')\over d(\tilde g g, g')^Q}|f(g')|dg' \lesssim\int_{d(\tilde g g, g')\leq\varepsilon^{-1}d(\tilde g, 0)+d(\tilde g, 0)} {1\over d(\tilde g g, g')^{Q-1}}|f(g')|dg' \\
&\lesssim \left( \varepsilon^{-1}d(\tilde g, 0)+d(\tilde g, 0)\right) M(f)(g)\\
&\lesssim \varepsilon^{-1}d(\tilde g, 0)M(f)(g).
\end{align*}

For $L_1$, 
\begin{align*}
|L_1|&=\left|\int_{d(g,g')>\varepsilon^{-1}d(\tilde g, 0)} K_\ell(g, g')\left(b(g)-b(\tilde g g)\right)f(g')dg' \right|\\
&=\left|b(g)-b(\tilde g g) \right|\left|\int_{d(g,g')>\varepsilon^{-1}d(\tilde g, 0)} K_\ell(g, g')f(g')dg' \right|\\
&\leq\left|b(g)-b(\tilde g g) \right|\sup_{t>0}\left|\int_{d(g,g')>t} K_\ell(g, g')f(g')dg' \right|
=\left|b(g)-b(\tilde g g) \right|\mathcal R^*_\ell(f)(g)\\
&\lesssim d(\tilde g, 0) \mathcal R^*_\ell(f)(g).
\end{align*}
Then by all the above estimates, we have
\begin{align*}
&\left( \int_{\mathcal G} \left|[b, \mathcal R_\ell]f(g)-\left([b,\mathcal G_\ell]f \right)_{B(g,r)} \right|^p w(g)dg\right)^{1\over p}\\
&=\left( \int_{\mathcal G} \left|{1\over |B(0,r)|}\int_{B(0,r)} [b, \mathcal R_\ell]f(g)-[b, \mathcal R_\ell]f(\tilde gg) d\tilde g \right|^p w(g)dg\right)^{1\over p}\\
&\leq \left( \int_{\mathcal G}{1\over |B(0,r)|^p}\left(\int_{B(0,r)}\left(|L_1| +|L_2|+|L_3|+|L_4|\right)d\tilde g\right)^{p} w(g)dg\right)^{1\over p}\\
 &\lesssim \left(\int_\mathcal G\big( r\mathcal R^*_\ell(f)(g)+\left( \varepsilon+\varepsilon^{-1}r\right)M(f)(g)\big)^p  w(g)dg\right)^{1\over p}\\
 &\lesssim r\left( \int_\mathcal G \left(R^*_\ell(f)(g)\right)^p w(g)dg\right)^{1\over p}+
 \left( \varepsilon+\varepsilon^{-1}r\right)\left( \int_\mathcal G \left(M(f)(g)\right)^p w(g)dg\right)^{1\over p}\\
 &\lesssim r\|f\|_{L^p_w(\mathcal G)}+ \left( \varepsilon+\varepsilon^{-1}r\right)\|f\|_{L^p_w(\mathcal G)},
\end{align*}
where the first term of the last inequality comes from \cite[Theorem 1.3]{HYY} together with \cite[Theorem 9.4.5]{Gr}.
Thus if we take $r<\varepsilon^2$, then 
$$\left( \int_{\mathcal G} \left|[b, \mathcal R_\ell]f(g)-\left([b,\mathcal G_\ell]f \right)_{B(g,r)} \right|^p w(g)dg\right)^{1\over p}\lesssim \varepsilon.$$
This shows that $[b, \mathcal R_\ell]E$ satisfies condition (iii) in Lemma \ref{lem r-k}. Hence, $[b, \mathcal R_\ell]$ is a compact operator. This finishes the proof of Theorem \ref{thm1}. 
\end{proof}


\section{Appendix: \  characterisation of ${\rm VMO}(\mathcal G)$}

In this section, we provide the characterisation of VMO space on stratified Lie groups
by giving the proof of Theorem \ref {thm-cmo}. We point out that the main frame of the  proof is similar to 
that in the Euclidean spaces from \cite{U1}. However, the technique in the proof in \cite{U1} 
there depends heavily on the decomposition cubes, and  on general stratified Lie groups, there is no such convenient tools. We have balls with respect to the metric instead of the cubes. Hence, the main contribution of our proof of Theorem \ref {thm-cmo} is to use balls to replace cubes in the Euclidean setting, which relies on the fact that the metric here is geometrically doubling and gives rise to the technique of coverings. 

We point out that our proof is written for stratified Lie groups but it also works for general space of homogeneous type in the sense of Coifman and Weiss \cite{CW} with the modification that we change $C_0^\infty(\mathcal G)$ to the Lipschitz function space on space of homogeneous type.

\begin{proof}[Proof of Theorem \ref{thm-cmo}]
In the following, for any integer $m$, we use $B^m$ to denote the ball $B(0, 2^m)$.

{\bf Necessary condition:} Assume that $f\in{\rm VMO}(\mathcal G)$. If $f\in C_0^\infty(\mathcal G)$, then (i)-(iii) hold. In fact, by the uniform continuity, $f$ satisfies (i).  Since $f\in L^1(\mathcal G)$,  $f$ satisfies (ii). By the fact that $f$ is compactly supported, $f$ satisfies (iii). If $f\in{\rm VMO}(\mathcal G)\setminus  C_0^\infty(\mathcal G)$, by definition, for any given $\varepsilon>0$, there exists $f_\varepsilon\in C_0^\infty(\mathcal G)$ such that $\|f-f_\varepsilon\|_{{\rm BMO}(\mathcal G)}<\varepsilon$. Since $f_\varepsilon$ satisfies (i)-(iii), by the triangle inequality of ${\rm BMO}(\mathcal G)$ norm, we can see (i)-(iii) hold for $f$.

{\bf Sufficient condition:}  In this proof for $j=1,2,\cdots, 8$, 
 the value 
$\alpha_j$ is a positive constant depending only on $Q$ and $\alpha_i$ for $1\leq i<j$. Assume that $f\in {\rm BMO}(\mathcal G)$ and satisfies (i)-(iii). To prove that $f\in{\rm VMO}(\mathcal G)$, it suffices to show that there exist  positive constants $\alpha_1$,  $\alpha_2$ 
 such that, for any $\varepsilon>0$, there exists $\phi_\varepsilon\in {\rm BMO}(\mathcal G)$ satisfying  
\begin{align}\label{g-h}
\inf_{h\in C_0^\infty(\mathcal G)}\|\phi_\varepsilon-h\|_{{\rm BMO}(\mathcal G)}<\alpha_1\varepsilon,
\end{align}
and
\begin{align}\label{g-f}
\|\phi_\varepsilon-f\|_{{\rm BMO}(\mathcal G)}<\alpha_2\varepsilon.
\end{align}

By (i), there exist $i_\varepsilon\in\mathbb N$ such that
\begin{align}\label{ieps}
\sup\left\{M(f,B): r_B\leq 2^{-i_\varepsilon+4}\right\}<\varepsilon.
\end{align}
By (iii), there exists $j_\varepsilon\in\mathbb N$ such that
\begin{align}\label{jeps}
\sup\left\{M(f,B): B\cap B^{j_\varepsilon}=\emptyset\right\}<\varepsilon.
\end{align}

We first establish a cover of $\mathcal G$.
 Observe that
\begin{align*}
B^{j_\varepsilon}=B^{-i_\varepsilon}\bigcup\left(\bigcup_{\nu=1}^{2^{j_\varepsilon+i_\varepsilon}-1}
B\left(0, (\nu+1)2^{-i_\varepsilon}\right)\setminus B\left(0, \nu2^{-i_\varepsilon}\right)
\right)=:\bigcup_{\nu=0}^{2^{j_\varepsilon+i_\varepsilon}-1}\mathcal R^{j_\varepsilon}_{\nu,-i_\varepsilon}
\end{align*}
For $m>j_\varepsilon$,
\begin{align*}\begin{split}
B^{m}\setminus B^{m-1}& =\bigcup_{\nu=0}^{2^{j_\varepsilon+i_\varepsilon-1}-1}
B\left(0, 2^{m-1}+ (\nu+1)2^{m-j_\varepsilon-i_\varepsilon}\right)\setminus B\left(0, 2^{m-1}+\nu2^{m-j_\varepsilon-i_\varepsilon}
\right)\\
&=:\bigcup_{\nu=0}^{2^{j_\varepsilon+i_\varepsilon-1}-1}\mathcal R^{m}_{\nu,m-j_\varepsilon-i_\varepsilon}.
\end{split}\end{align*}
For each $\mathcal R^{j_\varepsilon}_{\nu,-i_\varepsilon}$, $\nu=1,2,\cdots, 2^{j_\varepsilon+i_\varepsilon}-1$, let
$\tilde {\mathcal B}^{j_\varepsilon}_{\nu,-i_\varepsilon}$ be an open cover of $\mathcal R^{j_\varepsilon}_{\nu,-i_\varepsilon}$ consisting of open balls with radius $2^{-i_\varepsilon}$ and center on the sphere $S(0, (\nu+2^{-1})2^{-i_\varepsilon})$. Let $\mathcal B^{j_\varepsilon}_{0,-i_\varepsilon}=\{B(0,2^{-i_\varepsilon})\}$ and $\mathcal B^{j_\varepsilon}_{\nu,-i_\varepsilon}$ be the finite subcover of $\tilde {\mathcal B}^{j_\varepsilon}_{\nu,-i_\varepsilon}$. Similarly, for each $m>j_\varepsilon$ and $\nu=0,1,\cdots, 2^{j_\varepsilon+i_\varepsilon-1}-1$, let $\mathcal B^{m}_{\nu,m-j_\varepsilon-i_\varepsilon}$ be the finite cover of $\mathcal R^{m}_{\nu,m-j_\varepsilon-i_\varepsilon}$ consisting of open balls with radius $2^{m-j_\varepsilon-i_\varepsilon}$ and center on the sphere $S(0, (2^{m-1}+ (\nu+2^{-1})2^{m-j_\varepsilon-i_\varepsilon})$.

We define $B_g$ as follows.  If $g\in B^{j_\varepsilon}$, then there is $\nu\in\{0,1,\cdots, 2^{j_\varepsilon+i_\varepsilon}-1\}$ such that $g\in\mathcal R^{j_\varepsilon}_{\nu,-i_\varepsilon}$, let $B_g$ be a ball in $\mathcal B^{j_\varepsilon}_{\nu,-i_\varepsilon}$ that contains $g$. If $g\in B^m\setminus B^{m-1}$,  $m> j_\varepsilon $, then there is  $\nu\in\{ 0, 1, \cdots,  2^{j_\varepsilon+i_\varepsilon-1}-1\}$ such that $g\in \mathcal R^{m}_{\nu,m-j_\varepsilon-i_\varepsilon}$,  let $B_g$ be a ball in $\mathcal B^{m}_{\nu,m-j_\varepsilon-i_\varepsilon}$ that contains $g$.  We can see that if $\overline B_g\cap \overline B_{g'}\neq\emptyset$, then
\begin{align}\label{rbxy}
{\rm either}\ \  r_{B_g}\leq 2 ~ r_{B_{g'}}\ \  {\rm or}\ \  r_{B_{g'}}\leq 2~ r_{B_g}.
\end{align}
In fact, if $r_{B_g}>2r_{B_{g'}}$, then there is $m_0\in\mathbb N$ such that $g\in B^{m_0+2}\setminus B^{m_0+1}$ and $g'\in B^{m_0}$, thus 
$$d(g,g')\geq d(0,g)-d(0,g')\geq2^{m_0+1}-2^{m_0}>2^{m_0+2-j_\varepsilon-i_\varepsilon}+2^{m_0-j_\varepsilon-i_\varepsilon}=r_{B_g}+r_{B_{g'}},$$
 which is contradict to the fact that $\overline B_g\cap \overline B_{g'}\neq\emptyset$.

Now we define $\phi_\varepsilon$.
By (ii), there exists $m_\varepsilon>j_\varepsilon$ large enough such that when $r_B>2^{m_\varepsilon-i_\varepsilon-j_\varepsilon}$, we have
\begin{align}\label{lmeps}
M(f, B)<2^{Q(-i_\varepsilon-j_\varepsilon-1)-1}\varepsilon.
\end{align} 
Define
\begin{align*}
 \phi_\varepsilon(g)= 
 \begin{cases}
f_{B_g}, &{\rm if}\ \  g\in B^{m_\varepsilon},\\
 f_{B^{m_\varepsilon}\setminus B^{m_\varepsilon-1}},  \quad &{\rm if}\ \ g\in \mathcal G\setminus B^{m_\varepsilon},
 \end{cases} 
\end{align*}
where $f_B$ is defined in \eqref{fave}.

We claim that there exists a positive constant $\alpha_3, \alpha_4$ such that if $\overline B_g\cap \overline B_{g'}\neq\emptyset$ or $g,g'\in\mathcal G\setminus B^{m_\varepsilon-1}$, then
\begin{align}\label{gx-gy}
\left|\phi_{\varepsilon}(g)-\phi_{\varepsilon}(g') \right|<\alpha_3 \varepsilon.
\end{align}
And 
if $2B_{g}\cap 2B_{g'}\neq \emptyset$, then for any $g_1\in B_{g}$, $g_2\in B_{g'}$, we have 
\begin{align}\label{gx-gy1}
\left|\phi_{\varepsilon}(g_1)-\phi_{\varepsilon}(g_2) \right|<\alpha_4 \varepsilon.
\end{align}
Assume \eqref{gx-gy} and \eqref{gx-gy1} at the moment, we now continue to prove   the sufficiency  of  Theorem \ref{thm-cmo}.

Now we show \eqref{g-h}. Let
$\tilde h_{\varepsilon}(g):=\phi_{\varepsilon}(g)-f_{B^{m_\varepsilon}\setminus B^{m_\varepsilon-1}}.$
By definition of $\phi_{\varepsilon}$, we can see that
$\tilde h_{\varepsilon}(g)=0$ for $g\in\mathcal G\setminus B^{m_\varepsilon}$ and 
$\|\tilde h_{\varepsilon}-\phi_{\varepsilon}\|_{BMO(\mathcal G)}=0.$

Observe that $\supp (\tilde h_\varepsilon)\subset B^{m_\varepsilon}$ and there exists a function $h_\varepsilon\in C_c(\mathcal G)$ such that for any $g\in\mathcal G$,
$|\tilde h_{\varepsilon}(g)-h_{\varepsilon}(g)|<\varepsilon.$
Let $\psi\in C_0^\infty(\mathcal G)$ be a positive valued function with $\int_{\mathcal G}\psi=1$, then by \cite[Proposition 1.20]{FS}, $\psi_t\ast h_\varepsilon(g)$ approaches to  $h_\varepsilon(g)$ uniformly for $g\in\mathcal G$ as $t$ goes to $0$. Since
\begin{align*}
\|\psi_t\ast h_\varepsilon-\phi_\varepsilon\|_{{\rm BMO}(\mathcal G)}
&\leq \|\psi_t\ast h_\varepsilon-h_\varepsilon\|_{{\rm BMO}(\mathcal G)}
+\|h_\varepsilon-\tilde h_\varepsilon\|_{{\rm BMO}(\mathcal G)}+
\| \tilde h_\varepsilon-\phi_\varepsilon\|_{{\rm BMO}(\mathcal G)}\\
&\leq \|\psi_t\ast h_\varepsilon-h_\varepsilon\|_{{\rm BMO}(\mathcal G)}+2\varepsilon,
\end{align*}
we can obtain \eqref{g-h} by letting $t$ go to $0$ and by taking $\alpha_1=2$.

Now we show \eqref{g-f}. To this end, we only need to prove that for any ball $B\subset\mathcal G$,
$$M(f-\phi_\varepsilon, B)<\alpha_2 \varepsilon.$$
We first prove that for every $B_g$ with $g\in B^{m_\varepsilon}$,
\begin{equation}\label{Mfgx}
\int_{B_g}\left|f(g')-\phi_{\varepsilon}(g')\right|dg'\leq \alpha_5 \varepsilon |B_g|.
\end{equation}

In fact,
\begin{align*}
\int_{B_g}\left|f(g')-\phi_{\varepsilon}(g')\right|dg'
&=\int_{B_g\cap B^{m_\epsilon}}|f(g')-f_{B_{g'}}|dg'+ \int_{B_g\cap (\mathcal G\setminus B^{m_\epsilon})}|f(g')-f_{B^{m_\varepsilon}\setminus B^{m_\varepsilon-1}}|dg'.
\end{align*}

When 
 $g\in B(0, 2^{m_\varepsilon}-2^{m_\varepsilon-i_\varepsilon-j_\varepsilon})$, then 
$B_g\subset B^{m_\epsilon}$, thus
\begin{align*}
\int_{B_g}\left|f(g')-\phi_{\varepsilon}(g')\right|dg'&=\int_{B_g}|f(g')-f_{B_{g'}}|dg'\leq\int_{B_g}|f(g')-f_{B_g}|dg'+\int_{B_g}|f_{B_g}-f_{B_{g'}}|dg'\\
&= |B_g| M(f, B_g)+\int_{B_g}|f_{B_g}-f_{B_{g'}}|dg'.
\end{align*}
Note that if $g'\in B_g$, then $B_g\cap B_{g'}\neq\emptyset$.
Therefore, If $B_g\cap B^{j_\varepsilon}=\emptyset$, by \eqref{jeps} and \eqref {gx-gy}, we have
$$\int_{B_g}\left|f(g')-\phi_{\varepsilon}(g')\right|dg'<(\varepsilon+\alpha_3\varepsilon)|B_g|.$$
If $B_g\cap B^{j_\varepsilon}\neq\emptyset$, then $r_{B_g}\leq 2^{-i_\varepsilon+1}$, then by \eqref{ieps} and \eqref{gx-gy},
$$\int_{B_g}\left|f(g')-\phi_{\varepsilon}(g')\right|dg'<(\varepsilon+\alpha_3\varepsilon)|B_g|.$$

When $g\in B^{m_\varepsilon}\setminus B(0, 2^{m_\varepsilon}-2^{m_\varepsilon-j_\varepsilon-i_\varepsilon})$, it is clear that $B_g\cap B^{j_\varepsilon}=\emptyset$, then by \eqref{jeps}, \eqref{lmeps} and \eqref{gx-gy}, we have
\begin{align*}
&\int_{B_g}\left|f(g')-\phi_{\varepsilon}(g')\right|dg'\\
&\leq\int_{B_g\cap B^{m_\epsilon}}|f(g')-f_{B_g}|dg'+ \int_{B_g\cap B^{m_\epsilon}}|f_{B_g}-f_{B_{g'}}|dg'\\
&\quad+\int_{B_g\cap (\mathcal G\setminus B^{m_\epsilon})}|f(g')-f_{B^{m_\varepsilon+1}}|dg'+\int_{B_g\cap (\mathcal G\setminus B^{m_\epsilon})}|f_{B^{m_\varepsilon+1}}-f_{B^{m_\varepsilon}\setminus B^{m_\varepsilon-1}}|dg'\\
&\leq |B_g|M(f,B_g)+\alpha_3\varepsilon|B_g|+|B^{m_\varepsilon+1}|M(f,B^{m_\varepsilon+1})
+{|B^{m_\varepsilon+1}| |B_g|\over | B^{m_\varepsilon}\setminus B^{m_\varepsilon-1}|}M(f,B^{m_\varepsilon+1})\\
&<(2\varepsilon+\alpha_3\varepsilon)|B_g|.
\end{align*}
Then \eqref {Mfgx} holds by taking $\alpha_5=(2+\alpha_3)$.

Let $B$ be an arbitrary ball in $\mathcal G$, then
$M(f-\phi_\varepsilon, B)\leq M(f, B) + M(\phi_\varepsilon, B). $
If $B\subset B^{m_\varepsilon}$ and $\max\{r_{B_g}: B_g\cap B\neq\emptyset\}>8r_B$, then  
\begin{align}\label{min}
\min\{r_{B_g}: B_g\cap B\neq \emptyset\}>2r_B.
\end{align}

In fact, assume that $r_{B_{g_0}}=\max\{r_{B_g}: B_g\cap B\neq\emptyset\}$ and $g_0\in B^{l_0}\setminus B^{l_0-1}$ for some $l_0\in\mathbb Z$. 
Then $B\subset B^{l_0}\cap {3\over 2}B_{g_0}$.
If $l_0\leq j_\varepsilon$, then \eqref{min} holds.
If $l_0> j_\varepsilon$, then $r_{B_{g_0}}=2^{l_0-j_\varepsilon-i_\varepsilon}$, and
$$r_B<{1\over 8}r_{B_{g_0}}=2^{l_0-j_\varepsilon-i_\varepsilon-3}.$$
Since for any $g'\in{3\over 2}B_{g_0}$,
\begin{align*}
d(0,g')&\geq d(0,g_0)-d(g_0, g')\geq2^{l_0-1}-{3\over 2}2^{l_0-j_\varepsilon-i_\varepsilon}
>2^{l_0-1}-2^{l_0-j_\varepsilon-i_\varepsilon+1},
\end{align*}
we have
$$\operatorname{dist}(0,{3\over 2}B_{g_0}):=\inf_{g'\in {3\over 2}B_{g_0}}d(0,g')>2^{l_0-1}-2^{l_0-j_\varepsilon-i_\varepsilon+1}.$$
Thus $B\subset B^{l_0}\setminus {3\over 2}B^{l_0-2}$. Therefore, if $B_g\cap B\neq\emptyset$, then $g\in B^{l_0}\setminus B^{l_0-2}$, which implies that
$r_{B_g}\geq2^{l_0-2-j_\varepsilon-i_\varepsilon}>2r_B.$

From \eqref{min} we can see that  if $B_{g_i}\cap B\neq \emptyset$ and $B_{g_j}\cap B\neq \emptyset$, then $2B_{g_i}\cap 2B_{g_j}\neq \emptyset$.
Then by \eqref{gx-gy1}, we can get
\begin{align*}
M(\phi_\varepsilon, B)
&\leq{1\over |B|}\int_B {1\over |B|}\int_B \left| \phi_\varepsilon(g)- \phi_\varepsilon(g') \right|dg'dg\\
&= {1\over |B|^2}\sum_{i:B_{g_i}\cap B\neq\emptyset}\int_{B_{g_i}\cap B}
\sum_{j:B_{g_j}\cap B\neq\emptyset}\int_{B_{g_j}\cap B}\left| \phi_\varepsilon(g)- \phi_\varepsilon(g') \right|dg'dg\\
&<\alpha_4\varepsilon{1\over |B|^2}\left(\sum_{i:B_{g_i}\cap B\neq\emptyset}\left|B_{g_i}\cap B\right|\right)
\left(\sum_{i:B_{g_j}\cap B\neq\emptyset}\left|B_{g_j}\cap B\right|\right)<\alpha_4\alpha_6^2\varepsilon.
\end{align*}
Moreover, if $B\cap B^{j_\epsilon}\neq\emptyset$, then by \eqref{min}, $r_B<2^{-i_\varepsilon}$, thus by \eqref{ieps}, we have
$M(f, B)<\varepsilon.$
If $B\cap B^{j_\epsilon}=\emptyset$, then by \eqref{jeps}, $M(f, B)<\varepsilon.$
Consequently,
$$M(f-\phi_\varepsilon, B)\leq M(f, B) + M(\phi_\varepsilon, B)<\left(1+\alpha_4\alpha_6^2\right)\varepsilon.$$

If $B\subset B^{m_\varepsilon}$ and $\max\{r_{B_g}: B_g\cap B\neq\emptyset\}\leq 8r_B$, 
since the number of $B_g$ with $g\in B^{m_\varepsilon}$ that covers $B$ is bounded by $\alpha_7$,
 by  \eqref{Mfgx},  we have
\begin{align*}
M(f-\phi_\varepsilon, B)
&\leq {2\over |B|}\int_B \left|f(g)-\phi_\varepsilon(g)\right|dg\leq {2\over |B|} \sum_{i:B_{g_i}\cap B\neq\emptyset}\int_{B_{g_i}} \left|f(g)-\phi_\varepsilon(g)\right|dg\\
&\leq {2\over |B|}\alpha_5\varepsilon \sum_{i:B_{g_i}\cap B\neq\emptyset}\left|B_{g_i}\right|\leq {2\over |B|} \alpha_5\alpha_7\varepsilon|8B|=2^{3Q+1}\alpha_5\alpha_7\varepsilon.
\end{align*}

If $B\subset \mathcal G\setminus B^{m_\varepsilon-1}$, then $B\cap B^{j_\varepsilon}=\emptyset$, from \eqref{jeps} we can see $M(f, B)<\varepsilon$. By \eqref{gx-gy},
\begin{align*}
M(\phi_\varepsilon, B)
\leq{1\over |B|}\int_B {1\over |B|}\int_B \left|\phi_\varepsilon(g)-\phi_\varepsilon(g') \right|dg'dg
<\alpha_3\varepsilon.
\end{align*}
Therefore,
$$M(f-\phi_\varepsilon, B)\leq M(f, B)+M(\phi_\varepsilon, B)<(1+\alpha_3)\varepsilon.$$

If $B\cap (\mathcal G\setminus B^{m_\varepsilon})\neq \emptyset$ and $B\cap B^{m_\varepsilon-1}\neq \emptyset$. Let $p_{\scriptscriptstyle B}$ be the smallest integer such that $B\subset B^{p_{\scriptscriptstyle B}}$, then $p_{\scriptscriptstyle B}>m_\varepsilon$.
If $p_{\scriptscriptstyle B}=m_\varepsilon+1$, then 
$r_B>{1\over 2}(2^{m_\varepsilon}-2^{m_\varepsilon-1})=2^{m_\varepsilon-2}$. 
If $p_{\scriptscriptstyle B}>m_\varepsilon+1$, then
$r_B>{1\over 2}(2^{p_{\scriptscriptstyle B}-1}-2^{m_\varepsilon-1})$. Thus 
$${|B^{p_{\scriptscriptstyle B}}|\over |B|}\leq 2^{3Q}.$$
Therefore,
\begin{align*}
M(f-\phi_\varepsilon, B)
&\leq{1\over |B|}\int_B \left|f(g)-\phi_\varepsilon(g)-(f-\phi_\varepsilon)_{B^{p_{\scriptscriptstyle B}}} \right|dg
+\left| (f-\phi_\varepsilon)_{B^{p_{\scriptscriptstyle B}}}-(f-\phi_\varepsilon)_{B}\right|\\
&\leq2{|B^{p_{\scriptscriptstyle B}}|\over |B|}{1\over |B^{p_{\scriptscriptstyle B}}|}\int_{B^{p_{\scriptscriptstyle B}}} \left|f(g)-\phi_\varepsilon(g)-(f-\phi_\varepsilon)_{B^{p_{\scriptscriptstyle B}}} \right|dg\\
&\leq 2^{3Q+1}\left(M(f, B^{p_{\scriptscriptstyle B}})+M(\phi_\varepsilon, B^{p_{\scriptscriptstyle B}}) \right)\leq 2^{3Q+1}\left(\varepsilon+M(\phi_\varepsilon, B^{p_{\scriptscriptstyle B}}) \right),
\end{align*}
where the last inequality comes from \eqref{lmeps}.
By definition,
\begin{align*}
M(\phi_\varepsilon, B^{p_{\scriptscriptstyle B}})
&\leq {1\over |B^{p_{\scriptscriptstyle B}}|}\int_{B^{p_{\scriptscriptstyle B}}} \left|\phi_\varepsilon(g)-(\phi_\varepsilon)_{B^{p_{\scriptscriptstyle B}}\setminus B^{m_\varepsilon}} \right|dg
+\left|(\phi_\varepsilon)_{B^{p_{\scriptscriptstyle B}}\setminus B^{m_\varepsilon}}- (\phi_\varepsilon)_{B^{p_{\scriptscriptstyle B}}}\right|\\
&\leq{2\over |B^{p_{\scriptscriptstyle B}}|}\int_{B^{p_{\scriptscriptstyle B}}} \left|\phi_\varepsilon(g)-(\phi_\varepsilon)_{B^{p_{\scriptscriptstyle B}}\setminus B^{m_\varepsilon}} \right|dg.
\end{align*}
By \eqref{jeps}, \eqref{Mfgx} and the fact that $\phi_\varepsilon(g)=f_{B^{m_\varepsilon}\setminus B^{m_\varepsilon-1}}$ if $g\in\mathcal G\setminus B^{m_\varepsilon}$, we have
\begin{align*}
&\int_{B^{p_{\scriptscriptstyle B}}} \left|\phi_\varepsilon(g)-(\phi_\varepsilon)_{B^{p_{\scriptscriptstyle B}}\setminus B^{m_\varepsilon}} \right|dg\leq \int_{B^{p_{\scriptscriptstyle B}}}{1\over |B^{p_{\scriptscriptstyle B}}\setminus B^{m_\varepsilon}|}\int_{B^{p_{\scriptscriptstyle B}}\setminus B^{m_\varepsilon}}
|\phi_\varepsilon(g)-\phi_\varepsilon(g')|dg'dg\\
&= \int_{B^{m_\varepsilon}}
|\phi_\varepsilon(g)-f_{B^{m_\varepsilon}\setminus B^{m_\varepsilon-1}}|dg\\
&\leq \int_{B^{m_\varepsilon}}
|\phi_\varepsilon(g)-f(g)|dg+\int_{B^{m_\varepsilon}}
|f(g)-f_{B^{m_\varepsilon}}|dg+|B^{m_\varepsilon}||f_{B^{m_\varepsilon}}-f_{B^{m_\varepsilon}\setminus B^{m_\varepsilon-1}}|\\ 
&\leq \sum_{i:B_{g_i}\cap B^{m_\varepsilon}\neq\emptyset, g_i\in B^{m_\varepsilon}}
\int_{B_{g_i}}
|\phi_\varepsilon(g)-f(g)|dg+\left(|B^{m_\varepsilon}|+{|B^{m_\varepsilon}|^2\over |B^{m_\varepsilon}\setminus B^{m_\varepsilon-1}| }\right)M(f, B^{m_\varepsilon})\\
&<\alpha_5\varepsilon\sum_{i:B_{g_i}\cap B^{m_\varepsilon}\neq\emptyset, g_i\in B^{m_\varepsilon}}|B_{g_i}|+3\varepsilon\left|B^{m_\varepsilon}\right|<(\alpha_5\alpha_8+3)\varepsilon\left|B^{m_\varepsilon}\right|.
\end{align*}
Therefore,
\begin{align*}
M(f-\phi_\varepsilon, B)&\leq 2^{3Q+1}\left(\varepsilon+M(\phi_\varepsilon, B^{p_{\scriptscriptstyle B}}) \right)\leq2^{3Q+1}\left(\varepsilon+{2|B^{m_\varepsilon}|\over |B^{p_{\scriptscriptstyle B}}|} (\alpha_5\alpha_8+3)\varepsilon\right)\\
&<2^{3Q+1}\left( 2\alpha_5\alpha_8+7\right)\varepsilon.
\end{align*}
Then \eqref{g-f} holds by taking $\alpha_2=\max\{1+\alpha_4\alpha_6^2, 1+\alpha_3, 2^{3Q+1}( 2\alpha_5\alpha_8+7)\}$. This finishes the proof of Theorem \ref {thm-cmo}.

%

\medskip
{\it Proof of \eqref{gx-gy}:}
\medskip


 We first claim that 
\begin{align}\label{tgx-gy}
\sup\left\{\left|f_{B_g}-f_{B_{g'}}\right|: g, g'\in B^{m_\varepsilon}\setminus B^{m_\varepsilon-1} \right\}<\varepsilon.
\end{align}


By \eqref{lmeps}, for any $g\in B^{m_\varepsilon}\setminus B^{m_\varepsilon-1}$, we have
\begin{align*}
\left| f_{B_g}-f_{B^{m_{\varepsilon}+1}}\right|
&\leq
{|B^{m_{\varepsilon}+1}|\over |B_g|}{1\over |B^{m_{\varepsilon}+1}|}\int_{B^{m_{\varepsilon}+1}}\left|f(g')-f_{B^{m_{\varepsilon}+1}}\right|dg'\\
&={2^{Q(m_\varepsilon+1)}\over 2^{Q(m_\varepsilon-j_\varepsilon-i_\varepsilon)}}M(f, B^{m_\varepsilon +1})
<{\varepsilon\over 2}.
\end{align*}
Similarly, for any $g'\in B^{m_\varepsilon}\setminus B^{m_\varepsilon-1}$,
$|f_{B_{g'}}-f_{B^{m_{\varepsilon}+1}}|<{\varepsilon\over 2}.$
Consequently, \eqref{tgx-gy} holds.

For the case $g, g'\in\mathcal G\setminus B^{m_\varepsilon-1}$, firstly, if $g, g'\in\mathcal G\setminus B^{m_\varepsilon}$, then by definition 
$$\left|\phi_{\varepsilon}(g)-\phi_{\varepsilon}(g') \right|=0.$$
Secondly, if $g,g'\in B^{m_\varepsilon}\setminus B^{m_\varepsilon-1}$, then by \eqref{tgx-gy}, we have
$$\left|\phi_{\varepsilon}(g)-\phi_{\varepsilon}(g') \right|<\varepsilon.$$
Thirdly, without loss of generality, we may assume that $g\in B^{m_\varepsilon}\setminus B^{m_\varepsilon-1}$ and $g'\in\mathcal G\setminus B^{m_\varepsilon}$, then by \eqref{lmeps}, we have
\begin{align*}
\left|\phi_{\varepsilon}(g)-\phi_{\varepsilon}(g') \right|
&=\left|f_{B_g}-f_{B^{m_\varepsilon}\setminus B^{m_\varepsilon-1}}  \right|\leq \left|f_{B_g}-f_{B^{m_\varepsilon+1}}  \right|+\left|f_{B^{m_\varepsilon+1}}-f_{B^{m_\varepsilon}\setminus B^{m_\varepsilon-1}}  \right|\\
&\leq {|B^{m_\varepsilon+1}|\over |B_g|} M(f, B_{m_{\varepsilon+1}})+{|B^{m_\varepsilon+1}|\over |B^{m_\varepsilon}\setminus B^{m_\varepsilon-1}|}M(f, B_{m_{\varepsilon+1}})\\
&\leq \left( {2^{Q(m_\varepsilon+1)}\over 2^{Q(m_\varepsilon-j_\varepsilon-i_\varepsilon)}}
+{2^{Q(m_\varepsilon+1)}\over 2^{Qm_\varepsilon}-2^{Q(m_\varepsilon-1)}}\right)M(f, B_{m_{\varepsilon+1}})\\
&\leq \left( 2^{Q(1+j_\varepsilon+i_\varepsilon)}+2^{Q+1}\right)M(f, B_{m_{\varepsilon+1}})
<\varepsilon
\end{align*}

For the case $\overline B_g\cap \overline B_{g'}\neq\emptyset$ and  $g,g'\in B^{m_\varepsilon-1}$, we may assume $B_g\neq B_{g'}$ and $r_{B_g}\leq r_{B_{g'}}$. By \eqref{rbxy},
 $B_{g'}\subset 5B_g\subset 15B_{g'}$. 
 If $g'\in B^{j_\varepsilon+1}$, then by \eqref{ieps}, we have
\begin{align*}
\left|\phi_{\varepsilon}(g)-\phi_{\varepsilon}(g') \right|
&=\left|f_{B_g}-f_{B_{g'}}  \right|\leq \left|f_{B_g}-f_{3B_{g'}}  \right|+\left|f_{B_{g'}}-f_{3B_{g'}}  \right|\\
&\leq\left( {|3B_{g'}|\over |B_g|} +{|3B_{g'}|\over |B_{g'}|}\right)M(f, 3B_{g'})\leq \left(15^Q+3^Q\right)M(f, 3B_{g'})\\
&\leq\left(15^Q+3^Q\right)\varepsilon.
\end{align*}
If $g'\notin B^{j_\varepsilon+1}$, then $3B_{g'}\cap B^{j_\varepsilon}=\emptyset$, by \eqref{jeps}, we have
$$\left|\phi_{\varepsilon}(g)-\phi_{\varepsilon}(g') \right|\leq \left(3^Q+15^Q\right)M(f, 3B_{g'})\leq \left(15^Q+3^Q\right)\varepsilon.$$
Therefore, \eqref{gx-gy} holds by taking $\alpha_3=15^Q+3^Q$.

\medskip
{\it Proof of \eqref{gx-gy1}:}
\medskip


Since $g_1\in B_{g}$, $g_2\in B_{g'}$, we have $B_{g_1}\cap B_g\neq\emptyset$ and $B_{g_2}\cap B_{g'}\neq\emptyset$, by \eqref{gx-gy},
\begin{align*}
\left|\phi_{\varepsilon}(g_1)-\phi_{\varepsilon}(g_2) \right|
&\leq \left|\phi_{\varepsilon}(g_1)-\phi_{\varepsilon}(g) \right|+\left|\phi_{\varepsilon}(g)-\phi_{\varepsilon}(g') \right|+\left|\phi_{\varepsilon}(g')-\phi_{\varepsilon}(g_2) \right|\\
&\leq 2\alpha_3\varepsilon+\left|\phi_{\varepsilon}(g)-\phi_{\varepsilon}(g') \right|.
\end{align*}
We may assume $B_g\neq B_{g'}$ and $r_{B_g}\leq r_{B_{g'}}$.
If $g,g'\in\mathcal G\setminus B^{m_\varepsilon-1}$, then \eqref{gx-gy1} follows from \eqref{gx-gy}.
If $g,g'\in\ B^{m_\varepsilon-1}$, 
when $ g'\in B^{j_\varepsilon+1}$, then $2^{-i_\varepsilon}\leq r_{B_g}\leq r_{B_{g'}}\leq 2^{-i_\varepsilon+1}$, thus  $B_{g'}\subset 10B_g\subset 60B_{g'}$, by \eqref{ieps}, we have
\begin{align*}
\left|\phi_{\varepsilon}(g)-\phi_{\varepsilon}(g') \right|
&\leq \left|f_{B_g}-f_{6B_{g'}}  \right|+\left|f_{B_{g'}}-f_{6B_{g'}}  \right|=\left( {|6B_{g'}|\over |B_g|} +{|6B_{g'}|\over |B_{g'}|}\right)M(f, 6B_{g'})\\
&\leq \left(60^Q+6^Q\right)M(f, 6B_{g'})\leq\left(60^Q+6^Q\right)\varepsilon.
\end{align*}
When $g'\notin B^{j_\varepsilon+1}$, then there exist $\tilde m_0\in \mathbb N$ and $\tilde m_0\geq j_\varepsilon+2$ such that $g'\in B^{\tilde m_0}\setminus B^{\tilde m_0-1}$. Since $2B_{g}\cap 2B_{g'}\neq \emptyset$, we have $B_g\subset 6B_{g'}$. 
Note that $6B_{g'}\cap B^{\tilde m_0-2}=\emptyset$, (in fact, for any $\tilde g\in 6B_{g'}$, $d(0,\tilde g)\geq d(0, g')-d(g',\tilde g)\geq 2^{\tilde m_0-1}-6\cdot2^{\tilde m_0-j_\varepsilon-i_\varepsilon}>2^{\tilde m_0-2}$), thus $B_g\cap B^{\tilde m_0-2}=\emptyset$ and then ${1\over 2}r_{B_{g'}}=2^{\tilde m_0-1-j_\varepsilon-i_\varepsilon}\leq r_{B_g}\leq 2^{\tilde m_0-j_\varepsilon-i_\varepsilon}=r_{B_{g'}}$. Therefore, $B_{g'}\subset 10 B_g$. Then by \eqref{jeps}, we have
$$\left|\phi_{\varepsilon}(g)-\phi_{\varepsilon}(g') \right|\leq \left(60^Q+6^Q\right)M(f, 6B_{g'})<\left(60^Q+6^Q\right)\varepsilon.$$
If $g\in B^{m_\varepsilon-1}$ and $g'\in\mathcal G\setminus B^{m_\varepsilon-1}$, since $2B_{g}\cap 2B_{g'}\neq \emptyset$, by the construction of $B_g$ we can see that $g\in B^{m_\varepsilon-1}\setminus B^{m_\varepsilon-2}$ and $g'\in B^{m_\varepsilon}\setminus B^{m_\varepsilon-1}$. Thus, $B_{g'}\subset 10 B_{g}\subset 40 B_{g'}$. Then by \eqref{lmeps}, we have
$$\left|\phi_{\varepsilon}(g)-\phi_{\varepsilon}(g') \right|\leq \left(40^Q+4^Q\right)M(f, 4B_{g'})<\left(40^Q+4^Q\right)\varepsilon.$$
Taking $\alpha_4=60^Q+6^Q+2\alpha_3$, then \eqref{gx-gy1} holds.
%
%
%
%
\end{proof}

\bigskip
{\bf Acknowledgement:} P. Chen, X. T. Duong and J. Li are supported by ARC DP 160100153. P. Chen is also  supported by NNSF of China 11501583, Guangdong Natural Science Foundation 2016A030313351 and the Fundamental Research Funds for the Central Universities 161gpy45.
Q.Y. Wu is supported by NSF of China (Grants No. 11671185 and No. 11701250), the Natural Science Foundation of Shandong Province (No. ZR2018LA002), and the State Scholarship Fund of China (No. 201708370017).

\bibliographystyle{amsplain}

\begin{thebibliography}{10}


\bibitem{BLU} A. Bonfiglioli, E. Lanconelli, F. Uguzzoni, Stratied Lie Groups and Potential Theory for their
Sub-Laplacians. Springer Monographs in Mathematics 2007.

\bibitem{Cal}
A.-P. Calder\'on, 
Commutators of singular integral operators. 
Proc. Nat. Acad. Sci. U.S.A. {\bf53} (1965), 1092--1099. 

\bibitem{CRT} R. E. Castillo, J. C. Ramos Fern\'andez, and E. Trousselot,  Functions of bounded $(\varphi, p)$ mean oscillation, \textit{Proyecciones}, {\bf27} (2008), 163--177.


\bibitem{CCHTW} L. Chaffe, P. Chen, Y. Han, R. Torres and L. Ward, Characterization of compactness of commutators of bilinear singular integral operators, to appear in Proc. Amer. Math. Soc.

\bibitem{CRW}
{ R. Coifman, R. Rochberg and G. Weiss},
   Factorization theorems for Hardy spaces in several variables,
   \textit{Ann. of Math. (2)},
   \textbf{103}
   (1976), 611--635.
   
   
\bibitem{CW}
R. R. Coifman and G. Weiss, Extensions of Hardy spaces and their use in analysis, \textit{Bull. Amer. Math. Soc.}, {\bf83} (1977), 569--645.



\bibitem{DLLW}  {X. T. Duong, H.-Q. Li, J. Li and B. D. Wick, Lower bound for Riesz transform kernels and commutator theorems on stratified nilpotent Lie groups,} to appear in J. Math. Pure Appl.

\bibitem{DLLWW}  {X. T. Duong, H.-Q. Li, J. Li, B. D. Wick and Q. Y. Wu, Lower bound of Riesz transform kernels revisited and commutators on stratified Lie groups,}   arXiv:1803.01301.

\bibitem{DLMWY} {X. T. Duong, J. Li, S. Z. Mao, H. X. Wu and D. Y. Yang}, Compactness of Riesz transform commutator associated with Bessel operators, to appear in J. Anal. Math.


\bibitem {FS} G. B. Folland and E. M. Stein, Hardy Spaces on Homogeneous Groups, Princetion University Press,
Princeton, N. J., 1982.


\bibitem{GM} P. G\'orka and A. Macios,
The Riesz-Kolmogorov theorem on metric spaces,
\textit{Miskolc Math. Notes}, {\bf15} (2014),  459--465.


\bibitem{Gr} L. Grafakos, Classical and modern Fourier analysis. Pearson Education, Inc., Upper Saddle River, NJ, 2004. 


\bibitem{GWY} W. Guo, H. X. Wu and D. Y. Yang, A revisit on the compactness of commutators, arXiv:1712.08292.



\bibitem{HYY} G. Hu, Da. Yang and Do. Yang, Boundedness of maximal singular integral operators on spaces of homogeneous type and its applications, \textit{J. Math. Soc. Japan},  {\bf 59} (2007), 323--349.

\bibitem{HPR} T. Hyt\"onen, C. P\'erez and E. Rela, Sharp reverse H\"older property for $A_{\infty}$ weights on spaces of homogeneous type,  \textit{J. Funct. Anal.}, {\bf263} (2012), 3883--3899.




\bibitem{IMS} S. Indratno, D. Maldonado and S. Silwal,  A visual formalism for weights satisfying reverse inequalities,
 \textit{Expo. Math.}, {\bf 33} (2015),  1--29.



\bibitem{KL1} S. G. Krantz and S. Y. Li, Boundedness and compactness of integral operators on spaces of homogeneous type and applications I, \textit{J. Math. Anal. Appl.}, {\bf258} (2001), 629--641.


\bibitem{KL2} S. G. Krantz and S. Y. Li, Boundedness and compactness of integral operators on spaces of homogeneous type and applications II, \textit {J. Math. Anal. Appl.}, {\bf258} (2001), 642-657.





\bibitem{LNWW} J. Li, T. Nguyen, L.A. Ward and B.D. Wick, The Cauchy integral, bounded and compact commutators,
arXiv:1709.00703.



\bibitem{Mo} F. Montefalcone, Some remarks in differential and integral geometry of Carnot groups. PhD Thesis, University of Bologna, 2004.


\bibitem{Mu} B. Muckenhoupt, Weighted norm inequalities for the Hardy
maximal function,  \textit{Trans. Amer. Math. Soc.}, {\bf 165} (1972), 207--226.


\bibitem{Sa} L. Saloff-Coste,  Analyse sur les groupes de Lie \`a croissance polyn\^omiale, {\it Ark. Mat.}, {\bf28} (1990),  315--331.




\bibitem{TYY} J. Tao, Da. Yang and Do. Yang, Boundedness and compactness characterizations of Cauchy integral commutators on Morrey spaces,  arXiv:1801.04997.

\bibitem {To} A. Torchinsky, Real-variable methods in harmonic analysis, Pure and Applied Mathematics, vol. 123, Academic Press, Inc., Orlando, FL, 1986.

   
\bibitem{U1} A. Uchiyama, On the compactness of operators of Hankel type, \textit{T\^ohoku Math. J.}, {\bf30} (1978), 163-171.


\bibitem{VSC92} N. Th. Varopoulos, L. Saloff-Coste,T. Coulhon, Analysis and
geometry on groups, Cambridge Tracts in Mathematics, 100. Cambridge
University Press, Cambridge, 1992.










\end{thebibliography}

\end{document}